\title{\vspace{-0.6cm}Random directed graphs are robustly Hamiltonian}
\author{{Dan Hefetz \thanks{School of Mathematics, University of Birmingham, Edgbaston, Birmingham B15 2TT, United
Kingdom. Email: d.hefetz@bham.ac.uk. Research supported by EPSRC grant EP/K033379/1.}}
\quad {Angelika Steger \thanks{Institute of Theoretical Computer Science, ETH Z\"urich, CH-8092 Switzerland, Email: steger@inf.ethz.ch.}}
\quad {Benny Sudakov \thanks{
Department of Mathematics, ETH, 8092 Zurich.
Email: benjamin.sudakov@math.ethz.ch.
Research supported in part by SNSF grant 200021-149111 and
by a USA-Israel BSF grant.}}}

\documentclass[11pt]{article}
\usepackage{amsmath,amssymb,latexsym,color,epsfig,a4}
\usepackage{algorithm}
\usepackage{algorithmic}
\usepackage{float}
\usepackage{tikz}
\usetikzlibrary{shapes.geometric}
\usetikzlibrary{arrows}

\newif\ifnotesw\noteswtrue

\def\({\left(}
\def\){\right)}

\def\deg{\text{deg}}

\newtheorem{theorem}{Theorem}[section]
\newtheorem{lemma}[theorem]{Lemma}
\newtheorem{claim}[theorem]{Claim}
\newtheorem{proposition}[theorem]{Proposition}
\newtheorem{observation}[theorem]{Observation}
\newtheorem{corollary}[theorem]{Corollary}

\newtheorem{question}[theorem]{Question}
\newtheorem{remark}[theorem]{Remark}
\newtheorem{definition}[theorem]{Definition}

\renewcommand{\epsilon}{\varepsilon}

\newenvironment{proof}{\noindent{\bf Proof\,}}{\hfill$\Box$}

\date{}
\begin{document}

\maketitle

\begin{abstract}
A classical theorem of Ghouila-Houri from 1960 asserts that every directed graph on $n$ vertices with minimum out-degree and in-degree at least $n/2$ contains a directed Hamilton cycle. In this paper we extend this theorem to a random directed graph ${\mathcal D}(n,p)$, that is, a directed graph in which every ordered pair $(u,v)$ becomes an arc with probability $p$ independently of all other pairs. Motivated by the study of resilience of properties of random graphs, we prove that if $p \gg \log n/\sqrt{n}$, then a.a.s. every subdigraph of ${\mathcal D}(n,p)$ with minimum out-degree and in-degree at least $(1/2 + o(1)) n p$ contains a directed Hamilton cycle. The constant $1/2$ is asymptotically best possible. Our result also strengthens classical results about the existence of directed Hamilton cycles in random directed graphs.
\end{abstract}

\section{Introduction} \label{sec::intro}

A Hamilton cycle of a graph $G$ is a cycle which passes through every vertex of $G$ exactly once. A graph is said to be 
\emph{Hamiltonian} if it admits a Hamilton cycle. Hamiltonicity is one of the most central notions in graph theory, and has been 
intensively studied by numerous researchers for many years. The problem of deciding whether a graph is Hamiltonian or not is one of the 
NP-complete problems that Karp listed in his seminal paper~\cite{Karp}, and, accordingly, one cannot hope for a simple classification of 
such graphs. It is thus important to find generally applicable sufficient conditions for graphs to be Hamiltonian and in the last 60 years 
many interesting results were obtained in this direction. One of the first results of this type is the celebrated theorem of 
Dirac~\cite{Dirac}, asserting that every graph on $n \geq 3$ vertices with minimum degree at least $n/2$ (such graphs are called Dirac graphs) is Hamiltonian.

Dirac's Theorem provides a natural and easy to check sufficient condition for the Hamiltonicity of graphs with very high minimum degree. 
On the other hand, there are of course Hamiltonian graphs with minimum degree 2. Therefore, while Dirac's Theorem is sharp in general, 
one would like to have sufficient conditions for the Hamiltonicity of sparser graphs. When looking for such sufficient conditions, it is 
natural to consider random graphs with an appropriate edge probability. Erd\H{o}s and R\'enyi~\cite{ER} raised the question of what the 
threshold probability of Hamiltonicity in random graphs is. After a series of efforts by various researchers, including 
Korshunov~\cite{Korshunov} and P\'osa~\cite{Posa}, the problem was finally solved by Koml\'os and Szemer\'edi~\cite{KS} and 
independently by Bollob\'as~\cite{Bollobas}, who proved that if $p \geq (\log n + \log \log n + \omega(1))/n$, where $\omega(1)$ tends 
to infinity with $n$ arbitrarily slowly, then $G(n,p)$ is asymptotically almost surely (or a.a.s. for brevity) Hamiltonian. Note that 
this is best possible since for $p \leq (\log n + \log \log n - \omega(1))/n$ a.a.s. there are vertices of degree at most one in 
$G(n,p)$ (see, e.g.~\cite{BolBook}). An even stronger result was obtained by Bollob\'as~\cite{Bollobas}. He proved that for the random 
graph process, the hitting time for Hamiltonicity is exactly the same as the hitting time for having minimum degree 2, that is, a.a.s. 
the very edge which increases the minimum degree to 2 also makes the graph Hamiltonian.

In recent years there has been a lot of  interest in proving that certain families of graphs, like Dirac graphs or random graphs, are Hamiltonian in some robust sense. Results in this direction include showing that such graphs admit not only one Hamilton cycle but many (see, e.g.~\cite{CF, Janson, CK, GK}), that they admit many pairwise edge-disjoint Hamilton cycles (see, e.g.~\cite{KnKO, KSa, FK2, KLO, FKS}), that a player can claim all edges of a Hamilton cycle of these graphs, even when facing an optimal adversary (see, e.g.~\cite{SS, HKSS, BFHK, KLS, FGKN}), and many more.

The measure of how robust a graph is with respect to Hamiltonicity that we use in this paper is via the notion of \emph{local resilience}, which was introduced by Vu and the third author in~\cite{SV}. Let $G$ be a simple graph and let ${\mathcal P}$ be a monotone 
increasing graph property. The \emph{local resilience} of $G$ with respect to ${\mathcal P}$, denoted by $r_{\ell}(G, {\mathcal P})$, is the smallest non-negative integer $r$ such that one can obtain a graph which does not satisfy ${\mathcal P}$ by deleting at most $r$ edges at every vertex of $G$. Namely,

\vspace*{-0.7cm}

\begin{eqnarray*}
r_{\ell}(G, {\mathcal P}) = \min \{r : \exists H \subseteq G \textrm{ such that } \Delta(H) = r \textrm{ and } G \setminus H \notin 
{\mathcal P}\} \,.
\end{eqnarray*}     

\vspace*{-0.15cm}

Let ${\mathcal H}$ denote the graph property of being Hamiltonian. Using the notion of local resilience, one can restate the 
aforementioned results of Dirac~\cite{Dirac} as $r_{\ell}(K_n, {\mathcal H}) = \lfloor n/2 \rfloor = (1/2 + o(1)) n$. Following a series of results (see~\cite{SV, FK, BKS1, BKS2}), it was proved by Lee and Sudakov~\cite{LS} that a.a.s. $r_{\ell}(G(n,p), {\mathcal H}) = (1/2 + o(1)) n p$ for every $p \gg \log n/n$. This is a far reaching generalization of Dirac's Theorem, since a complete graph is also a random graph $G(n,p)$ with $p=1$.
 
In this paper we aim to prove analogous results for directed graphs (or digraphs for brevity). Similarly to the case of undirected 
graphs, we define the local resilience of a digraph $D$ with respect to a monotone increasing digraph property ${\mathcal P}$ to be the 
smallest non-negative integer $r$ such that one can obtain a digraph which does not satisfy ${\mathcal P}$ by deleting at most $r$ 
out-going arcs and at most $r$ in-going arcs at every vertex of $D$. Namely,
\begin{eqnarray*}
r_{\ell}(D, {\mathcal P}) = \min \{r : \exists H \subseteq D \textrm{ such that } \Delta^+(H) \leq r, \Delta^-(H) \leq r \textrm{ and } 
D \setminus H \notin {\mathcal P}\} \,.
\end{eqnarray*}

\vspace*{-0.15cm}

For a positive integer $n$ and $0 \leq p = p(n) \leq 1$, let ${\mathcal D}(n,p)$ denote the probability space of random labeled 
\emph{directed} graphs with vertex set $[n] := \{1,2, \ldots, n\}$. That is, for every ordered pair $(u,v)$ with $1 \leq u \neq v \leq 
n$ we flip a coin, all coin flips being mutually independent. With probability $p$ we include the arc $(u,v)$ in our digraph and with probability $1-p$ we do not. By abuse of notation we sometimes use ${\mathcal D}(n,p)$ to denote a single element of the space ${\mathcal D}(n,p)$. We also use ${\mathcal H}$ to denote the \emph{digraph} property of being Hamiltonian, that is, a digraph $D$ satisfies $D \in {\mathcal H}$ if and only if $D$ admits a directed Hamilton cycle.

Similarly to the aforementioned results of Koml\'os and Szemer\'edi and of Bollob\'as regarding the Hamiltonicity of random undirected 
graphs, results of McDiarmid~\cite{McDiarmid} and Frieze~\cite{Frieze} imply that a.a.s. $\mathcal D(n,p)$ is Hamiltonian for every $p \geq (1 + o(1)) \log n /n$ but admits no Hamilton cycles when $p \leq (1 - o(1)) \log n /n$. Moreover, a classical analog of Dirac's Theorem for directed graphs was proved in 1960 by Ghouila-Houri~\cite{GH}. It asserts that every directed graph on $n$ vertices with minimum out-degree and in-degree at least $n/2$ contains a directed Hamilton cycle. Stating this result in terms of local resilience we have that $r_{\ell}({\mathcal D(n,1)}, {\mathcal H}) = \lfloor n/2 \rfloor$. Hence it is natural to ask whether one can generalize the theorem of Ghouila-Houri to sparse random directed graphs, similar to the generalization of Dirac's Theorem proved by Lee and Sudakov in~\cite{LS}. In this paper we obtain such a result for every $p$ which is not too small.

\begin{theorem} \label{directedHam}
For every fixed $\alpha > 0$, if the arc probability of the random digraph ${\mathcal D}(n,p)$ satisfies $\log n/\sqrt{n} \ll p = p(n) \leq 1$, then a.a.s.
$$
(1/2 - \alpha) n p \leq r_{\ell}({\mathcal D}(n,p), {\mathcal H}) \leq (1/2 + \alpha) n p.
$$
\end{theorem}

\subsection{Notation and preliminaries}

\noindent For the sake of simplicity and clarity of presentation, we do not make a par\-ti\-cu\-lar effort to optimize some of the 
constants obtained in our proofs. We also omit floor and ceiling signs whenever these are not crucial. Most of our results are 
asymptotic in nature and whenever necessary we assume that $n$ is sufficiently large. Throughout the paper, $\log$ stands for the 
natural logarithm, unless explicitly stated otherwise. We say that a graph property ${\mathcal P}$ holds \emph{asymptotically almost 
surely}, or a.a.s. for brevity, if the probability of satisfying ${\mathcal P}$ tends to 1 as the number of vertices $n$ tends to infinity. Our graph-theoretic notation is standard and follows that of~\cite{West}. In particular, we use the following.

For a directed graph (or \emph{digraph} for brevity) $D$, let $V(D)$ and $E(D)$ denote its sets of vertices and arcs respectively, and 
let $v(D) = |V(D)|$ and $e(D) = |E(D)|$. For a set $A \subseteq V(D)$, let $E_D(A)$ denote the set of arcs of $D$ with both endpoints in 
$A$, and let $e_D(A) = |E_D(A)|$. For disjoint sets $A,B \subseteq V(D)$, let $E_D(A,B)$ denote the set of arcs of $D$ which are 
oriented from some vertex of $A$ towards some vertex of $B$, and let $e_D(A,B) = |E_D(A,B)|$. Let $d_D(A,B) = \frac{e_D(A,B)}{|A||B|}$ 
denote the \emph{directed density} of the ordered pair $(A,B)$ in $D$. For a vertex $u \in V(D)$ and a set $Y \subseteq V(D)$ let 
$N^+_D(u,Y) = \{y \in Y : (u,y) \in E(D)\}$ denote the set of \emph{out-neighbors} of $u$ in $Y$, and let $\deg^+_D(u,Y) = 
|N^+_D(u,Y)|$. Similarly, let $N^-_D(u,Y) = \{y \in Y : (y,u) \in E(D)\}$ denote the set of \emph{in-neighbors} of $u$ in $Y$, and let 
$\deg^-_D(u,Y) = |N^-_D(u,Y)|$. Let $\deg_D(u,Y) = \deg^+_D(u,Y) + \deg^-_D(u,Y)$. We abbreviate $\deg^+_D(u,V(D))$ to $\deg^+_D(u)$, 
$\deg^-_D(u,V(D))$ to $\deg^-_D(u)$, and $\deg_D(u,V(D))$ to $\deg_D(u)$. Moreover, $\deg^+_D(u)$ is referred to as the out-degree of 
$u$, $\deg^-_D(u)$ is referred to as the in-degree of $u$, and $\deg_D(u)$ is referred to as the degree of $u$. Let $\delta^+(D) = 
\min_{u \in V(D)} \deg^+_D(u), \; \delta^-(D) = \min_{u \in V(D)} \deg^-_D(u)$ and $\delta(D) = \min_{u \in V(D)} \deg_D(u)$ denote the 
minimum out-degree of $D$, the minimum in-degree of $D$, and the minimum degree of $D$, respectively. Similarly, let $\Delta^+(D) = 
\max_{u \in V(D)} \deg^+_D(u), \; \Delta^-(D) = \max_{u \in V(D)} \deg^-_D(u)$ and $\Delta(D) = \max_{u \in V(D)} \deg_D(u)$ denote the 
maximum out-degree of $D$, the maximum in-degree of $D$, and the maximum degree of $D$, respectively. For disjoint sets $A, B \subseteq 
V(D)$, let $N^+_D(A,B) = \bigcup_{a \in A} N^+_D(a,B)$, and similarly let $N^-_D(A,B) = \bigcup_{a \in A} N^-_D(a,B)$. Sometimes, if 
there is no risk of confusion, we discard the subscript $D$ in the above notation.

Throughout the paper we will make use of the following well-known bounds on the lower and upper tails of the binomial distribution due to Chernoff (see e.g.~\cite{JLR}).
\begin{theorem}[Chernoff bounds] \label{th::Chernoff}
Let $X \sim Bin(n,p)$, let $0 \leq \varepsilon \leq 1$ and let $x \geq 7 n p$. Then
\begin{description}
\item [(i)] $Pr (X \leq (1-\varepsilon) n p) \leq \exp \left\{-\frac{\varepsilon^2 n p}{2} \right\}$.
\item [(ii)] $Pr (X \geq (1+\varepsilon) n p) \leq \exp\left\{-\frac{\varepsilon^2 n p}{3} \right\}$.
\item [(iii)] $Pr (|X - n p| \geq \varepsilon n p) \leq 2 \exp\left\{-\frac{\varepsilon^2 n p}{3} \right\}$.
\item [(iv)] $Pr(X \geq x) \leq e^{-x}$.
\end{description}
\end{theorem}

The rest of this paper is organized as follows. In Section~\ref{sec::outline} we briefly outline some of the main ideas of our proof of Theorem~\ref{directedHam}. In Section~\ref{sec::directRegularity} we discuss some tools that will be used in the proof of Theorem~\ref{directedHam}, most notably, the sparse diregularity Lemma. In Section~\ref{sec::randomDigraphs} we prove various properties of random directed graphs that will be used in the proof of Theorem~\ref{directedHam}. In Section~\ref{sec::proofMain} we prove Theorem~\ref{directedHam}. Finally, in Section~\ref{sec::openprob} we present some open problems.

\section{A short outline of the proof of our main result}
\label{sec::outline}
Since our proof of the lower bound in Theorem~\ref{directedHam} is quite involved, we first sketch very briefly some of its main ideas. Some of the concepts and tools we use, will only be stated precisely and proved in the following sections. We start by considering the case of constant $p$ and then describe some of the additional difficulties which arise if one allows $p = o(1)$.

Let $D \in {\mathcal D}(n,p)$ and let $D'=(V,E)$ be a digraph obtained from $D$ by deleting at each vertex $u \in V(D)$ at most $(1/2 - \alpha)\deg_D^+(u)$ out-going arcs and at most $(1/2 - \alpha)\deg_D^-(u)$ in-going arcs. Note that both the out-degree and the in-degree in $D$ of every vertex $u \in V$ is a.a.s. roughly $n p$ and therefore both the out-degree and the in-degree in $D'$ of every vertex $u \in V$ is a.a.s. at least $(1/2 + \alpha - o(1)) n p$.

Apply the Directed Regularity Lemma (see Section~\ref{sec::directRegularity} for more details) to $D'=(V,E)$ with appropriate parameters. Let $\{V_0, V_1, \ldots, V_k\}$ be the corresponding $\varepsilon$-regular partition. For some appropriately chosen $\delta > 0$ let $R = R(D', \delta)$ be the corresponding regularity digraph; that is, the directed graph with vertex set $\{v_1, v_2, \ldots, v_k\}$ such that for every $1 \leq i \neq j \leq k$, $(v_i, v_j)$ is an arc of $R$ if and only if $(V_i, V_j)$ is $\varepsilon$-regular with directed density at least $\delta$.

We claim that $R$ contains a subgraph $R'$ on at least $(1 - \beta)k$ vertices and with minimum out-degree and in-degree at least $|R'|/2$. Indeed, suppose not. Then by recursively deleting vertices whose out-degree or in-degree is strictly smaller than half the number of vertices, we would delete at least $\beta k$ vertices. By symmetry we may assume that half of them have too small an in-degree. By assuming that $\varepsilon$ is sufficiently small compared to $\beta$ we see that the majority of these missing arcs correspond to $\varepsilon$-regular pairs and thus have to have density less than~$\delta$. If we now assume that $\beta$ is sufficiently small compared to $\alpha$, one can check that it follows that in order to obtain $D'$ from $D$ we have deleted strictly more  than $(1/2 - \alpha) \deg_D^-(u)$ in-going arcs at some vertex $u \in V(D)$, contrary to our assumption. By Ghouila-Houri's Theorem~\cite{GH} (see Theorem~\ref{th::CorOfGH}) the subgraph $R'$ is Hamiltonian. Equivalently, there exists an almost spanning cycle $C_R : v_1, v_2, \ldots, v_r, v_1$ of $R$. This corresponds to a directed ``cycle'' $C : V_1, V_2, \ldots, V_r, V_1$ of $D'$. Note that, by the definition of $R$, the pair $(V_i, V_{(i \mod r) + 1})$ is $\varepsilon$-regular with positive directed density for every $1 \leq i \leq r$.

In order to build a directed Hamilton cycle of $D'$, we will first build an almost spanning cycle $C_1$ and then absorb all the remaining vertices. In order to add some vertex $u$ to $C_1$ we will find an arc $(x,y) \in E(C_1)$ such that $(x,u) \in E(D')$ and $(u,y) \in E(D')$ and will then remove $(x,y)$ from $C_1$ and add to it $(x,u)$ and $(u,y)$. In order for this to work, when building $C_1$ we will have to ensure that there exists a pairing (in the sense suggested above) of all vertices of $V \setminus V(C_1)$ with certain arcs of $C_1$; we refer to this as our \emph{main task}. From now on we focus on building $C_1$.

We build $C_1$ by continuously moving around $C$ until we nearly exhaust all of the sets $V_i$. We always choose only \emph{nice} vertices, that is, vertices that have roughly the right number of out-neighbors in the next set along $C$. Thinking ahead to the moment at which we will want to close the directed path we are building into a directed cycle, we start building $C_1$ at a vertex $v_0 \in V_1$ which, in addition to being nice, is also \emph{backwards nice}, that is, it has roughly the right number of in-neighbors in $V_r$. We refrain from touching some predetermined subset of those in-neighbors until we attempt to close the directed path we built into a directed cycle.

In the digraph $D$ a \emph{typical} vertex $u$ has roughly $|V_i| p$ in-neighbors and roughly $|V_i| p$ out-neighbors in $V_i$ for every $1 \leq i \leq r$. A simple calculation shows that a.a.s. there is only a very small number of atypical vertices. Our first task is to build a directed path which includes all of these atypical vertices. Let $u$ be an arbitrary atypical vertex which we wish to add to the path we have built thus far. It follows by the aforementioned lower bounds on the minimum in-degree and the minimum out-degree of $D'$ that there must exist indices $1 \leq j_1 \neq j_2 \leq r$ such that $u$ has many in-neighbors in $V_{j_1}$ and many out-neighbors in $V_{j_2}$. We walk along $C$ (always choosing nice vertices as described above) until we reach $V_{j_1}$. Using regularity we can ensure that we enter $V_{j_1}$ in an in-neighbor of $u$. We can thus add $u$ to the path and proceed to a nice vertex of $V_{j_2}$. Once all atypical vertices are included in the path we focus on our main task.

We continue moving along $C$ as before except that, at every step, the new vertex we add to the path is chosen uniformly at random from all nice vertices. Using this randomness we wish to show that for every vertex $u$ which will not be included in $C_1$, there will be many times in which we claim an in-neighbor of $u$ followed by an out-neighbor of $u$; each such time is referred to as a \emph{successful trial}. In our analysis we use known bounds on the tail of the binomial distribution. Hence, in order to ensure the independence of trials which is needed for the binomial distribution and in order to bound from below the probability that a single trial is successful, we will only consider arcs of $C_1$ which are far from each other, as trials for each specific vertex $u \in V \setminus V(C_1)$.

Once the path we built is sufficiently long, we close it into a cycle. In order to do so we simply walk along the cycle $C$ until we reach $V_{r-2}$. Using regularity we can now extend the path by two more arcs such that the second arc touches some $x \in V_r$ which is an in-neighbor of $v_0$. Claiming $(x,v_0)$ completes the cycle $C_1$.

Our random procedure for building (the main part of) the directed path (see above) ensures that a.a.s. there will be strictly more than $|V \setminus V(C_1)|$ successful trials for every vertex $u \in V \setminus V(C_1)$. We can therefore greedily add all the vertices of $V \setminus V(C_1)$ to our directed cycle.

When $p = o(1)$ we use a sparse version of the Directed Regularity Lemma (see Lemma~\ref{diregularity}). The main difficulty in applying this lemma is that now $p \ll \varepsilon$. Hence, we have to ensure regularity of very small sets. This is done in Proposition~\ref{regularSmallSets} and Corollary~\ref{cor::regularityOfSmallPairs}. However we cannot ensure regularity of all small sets. This leads to the introduction of more atypical vertices (see the definitions of bad vertices of type I and II in Section~\ref{sec::randomDigraphs}). We will include some of these vertices in the initial segment of our directed path while intentionally avoiding others until the final stage of absorbing all remaining vertices. Since there are now many atypical vertices and, on the other hand, the size of a neighborhood of a vertex is very small, we might end up exhausting the neighborhood of some vertex already at an early stage. If such a vertex is not included in $C_1$ we might not be able to add it to the cycle. In order to avert this danger we will include all such vertices in the initial segment of our directed path as well. While our main task remains essentially the same, our success in fulfilling it will be much more limited. In particular, the number of successful trials per vertex of $V \setminus V(C_1)$ will be strictly smaller than $|V \setminus V(C_1)|$. Hence, in order to obtain the required pairing of all vertices of $V \setminus V(C_1)$ with certain arcs of $C_1$, we will use Hall's Theorem and the fact that $D'$ is a subdigraph of a random directed graph.

\section{The Sparse Diregularity Lemma}
\label{sec::directRegularity}

The Sparse Diregularity Lemma, due to Kohayakawa~\cite{Kohayakawa}, is a version of Szemer\'edi's Regularity Lemma (see~\cite{Szemeredi}) for sparse directed graphs. Before stating the lemma, we introduce the relevant terminology. Let $H$ be a directed bipartite graph with bipartition $V(H) = A \cup B$, let $0 < p \leq 1$ and let $\varepsilon > 0$. We say that the \emph{ordered} pair $(A,B)$ is \emph{$(\varepsilon, p)$-regular} if $|d_H(X,Y) - d_H(A,B)| \leq \varepsilon p$ holds for every $X \subseteq A$ and every $Y \subseteq B$ such that $|X| \geq \varepsilon |A|$ and $|Y| \geq \varepsilon |B|$. If $(A,B)$ is $(\varepsilon, d)$-regular, where $d = d_H(A,B)$, then we say that $(A,B)$ is $(\varepsilon)$-regular. The following two observations follow directly from the definition of $(\varepsilon, p)$-regularity.

\begin{observation} \label{obs::regularity1}
Let $\varepsilon_1 \leq \varepsilon_2$ and $p_1 \leq p_2$. If $(A,B)$ is $(\varepsilon_1, p_1)$-regular, then it is also $(\varepsilon_2, p_2)$-regular.  
\end{observation}

\begin{observation} \label{obs::regularity2}
If $(A,B)$ is $(\varepsilon, p)$-regular and $d \leq p$ is a real number (which might depend on $A$ and $B$), then $(A,B)$ is also $(\varepsilon p/d, d)$-regular.
\end{observation}      

Let $D=(V,E)$ be a digraph. A partition $\{V_0, V_1, \ldots, V_k\}$ of $V$ in which the, possibly empty, set $V_0$ has been singled out as an exceptional set, is called an $(\varepsilon, p)$-\emph{regular partition} if it satisfies the following conditions:
\begin{description}
\item [$(i)$] $|V_0| \leq \varepsilon |V|$;
\item [$(ii)$] $|V_1| = \ldots = |V_k|$;
\item [$(iii)$] all but at most $\varepsilon k^2$ of the pairs $(V_i, V_j)$, where $1 \leq i \neq j \leq k$, are $(\varepsilon, p)$-regular.
\end{description}

\begin{remark} \label{fewLowDegVertices}
It follows from Property $(iii)$ above that, if $\{V_0, V_1, \ldots, V_k\}$ is an $(\varepsilon, p)$-\emph{regular partition}, then there are at most $\sqrt{\varepsilon}k$ indices $1 \leq i \leq k$ for which there are at least $\sqrt{\varepsilon}k$ indices $1 \leq j \neq i \leq k$ such that $(V_i,V_j)$ is not $(\varepsilon, p)$-regular. Similarly, there are at most $\sqrt{\varepsilon}k$ indices $1 \leq i \leq k$ for which there are at least $\sqrt{\varepsilon}k$ indices $1 \leq j \neq i \leq k$ such that $(V_j, V_i)$ is not $(\varepsilon, p)$-regular. 
\end{remark}

Let $0 < \eta, p \leq 1$ and $L > 1$ be real numbers. A digraph $D=(V,E)$ is said to be \emph{$(\eta, L, p)$-bounded} if $e_D(A,B) \leq L p |A| |B|$ holds for every pair of disjoint sets $A,B \subseteq V$ such that $|A|, |B| \geq \eta |V|$.

\begin{remark} \label{$D(n,p)$IsBounded}
Let $n$ be a positive integer and let $p = p(n) \gg 1/n$. It is easy to see that ${\mathcal D}(n,p)$ is a.a.s. $(\eta, L, p)$-bounded for any fixed $0 < \eta \leq 1$ and $L > 1$.
\end{remark}

\begin{lemma} [Sparse Diregularity Lemma] \label{diregularity}
For every positive integer $m$, and every real numbers $\varepsilon > 0$ and $L > 1$, there exist integers $n_0 = n_0(m, \varepsilon, L)$ and $M = M(m, \varepsilon, L) \geq m$ and a real number $0 < \eta = \eta(m, \varepsilon, L) \leq 1$, such that for every $0 < p \leq 1$, every $(\eta, L, p)$-bounded digraph of order $n \geq n_0$ admits an $(\varepsilon, p)$-regular partition $\{V_0, V_1, \ldots, V_k\}$, where $m \leq k \leq M$.
\end{lemma}

Let $D=(V,E)$ be a directed graph and let $\delta > 0$ be a parameter. Given an $(\varepsilon, p)$-regular partition $\{V_0, V_1, \ldots, V_k\}$ of $V$, we define the \emph{regularity digraph} $R = R(D, \delta)$ to be the directed graph with vertex set $\{v_1, v_2, \ldots, v_k\}$ such that for every $1 \leq i \neq j \leq k$, $(v_i, v_j)$ is an arc of $R$ if and only if $(V_i, V_j)$ is $(\varepsilon, p)$-regular with directed density at least $\delta$.

Note that two ordered $(\varepsilon, p)$-regular pairs $(V_i, V_{i'})$ and $(V_j, V_{j'})$ in the partition might have different directed densities. While this is not really a problem, it would be convenient to assume that all ordered regular pairs with positive density have the same directed density. This can be done by applying the following lemma from~\cite{GS}.

\begin{lemma} \label{lem::GS}
For every $0 < \varepsilon \leq 1/6$ there exists a constant $C = C(\varepsilon)$ such that any $(\varepsilon)$-regular graph $H = (A \cup B, E)$ contains a $(2\varepsilon)$-regular subgraph $H' = (A \cup B, E')$ with $|E'| = m$ edges for all $m$ satisfying $C |V(H)| \leq m \leq |E|$. 
\end{lemma}

The following proposition asserts that most small subsets of regular pairs are also regular.
\begin{proposition} \label{regularSmallSets}
For every $0 < \beta, \varepsilon' < 1$, there exist $\varepsilon_0 = \varepsilon_0(\beta, \varepsilon') > 0$ and $C = C(\varepsilon')$ such that, for any $0 < \varepsilon \leq \varepsilon_0$, the following holds. Suppose that $D = (V_1 \cup V_2, E)$ is a bipartite digraph such that $(V_1, V_2)$ is $(\varepsilon)$-regular with directed density $d = d_D(V_1, V_2)$, and suppose that $q \geq C d^{-1}$ is an integer. Then the number of sets $Q \subseteq V_1$ of size $q$ that contain a set $\tilde{Q}$ of size at least $(1 - \varepsilon')q$ for which $(\tilde{Q}, V_2)$ is $(\varepsilon')$-regular with directed density $d'$ satisfying $(1 - \varepsilon)d \leq d' \leq (1 + \varepsilon)d$ is at least $(1 - \beta^q) \binom{|V_1|}{q}$. Similarly, the number of sets $Q \subseteq V_2$ of size $q$ that contain a set $\tilde{Q}$ of size at least $(1 - \varepsilon')q$ for which $(V_1, \tilde{Q})$ is $(\varepsilon')$-regular with directed density $d'$ satisfying $(1 - \varepsilon)d \leq d' \leq (1 + \varepsilon)d$ is at least $(1 - \beta^q) \binom{|V_2|}{q}$.
\end{proposition}

The analogous statement for undirected graphs was proved in~\cite{GKRS}. Since we only care about the arcs of $D$ oriented from $V_1$ to $V_2$, Proposition~\ref{regularSmallSets} is equivalent to Theorem 3.7 from~\cite{GKRS}.

In Section~\ref{sec::randomDigraphs} we will need the following corollary of Proposition~\ref{regularSmallSets}.

\begin{corollary} \label{cor::regularityOfSmallPairs}
For every $0 < \beta, \varepsilon' < 1$, there exist $\varepsilon_0 = \varepsilon_0(\beta, \varepsilon') > 0$ and $C = C(\beta, \varepsilon')$ such that, for any $0 < \varepsilon \leq \varepsilon_0$, the following holds. Suppose that $D = (V_1 \cup V_2, E)$ is a bipartite digraph such that $(V_1, V_2)$ is $(\varepsilon)$-regular with directed density $d = d_D(V_1, V_2)$, and suppose that $q_1, q_2 \geq C d^{-1}$ are integers. Then the number of pairs $(Q_1, Q_2)$ such that 
\begin{description}
\item [(i)] $Q_1 \in \binom{V_1}{q_1}$, $Q_2 \in \binom{V_2}{q_2}$ and 
\item [(ii)] there exist $\tilde{Q}_1 \subseteq Q_1$ of size at least $(1 - \varepsilon') q_1$ and $\tilde{Q}_2 \subseteq Q_2$ of size at least $(1 - \varepsilon') q_2$ such that the pair $(\tilde{Q}_1, \tilde{Q}_2)$ is $(\varepsilon')$-regular with directed density $d''$ satisfying $(1 - \varepsilon')^2 d \leq d'' \leq (1 + \varepsilon')^2 d$ 
\end{description}
is at least $(1 - \beta^{q_1} - \beta^{q_2}) \binom{|V_1|}{q_1} \binom{|V_2|}{q_2}$.
\end{corollary}

\begin{proof}
Let $\varepsilon_0^1 = \varepsilon_0^1(\beta, \varepsilon') > 0$ and $C_1 = C_1(\varepsilon')$ be the constants whose existence follows from Proposition~\ref{regularSmallSets} and let $\varepsilon_1 = \min \{\varepsilon_0^1, \varepsilon'\}$. Let $\varepsilon_0^2 = \varepsilon_0^2(\beta, \varepsilon_1) > 0$ and $C_2 = C_2(\varepsilon_1)$ be the constants whose existence follows from Proposition~\ref{regularSmallSets}. Let $C = \max \{C_1, C_2\}$, let $\varepsilon_0 = \min \{\varepsilon_0^2, \varepsilon_1\}$ and fix some $0 < \varepsilon \leq \varepsilon_0$.

Let ${\mathcal F}_1$ denote the family of all sets $B \subseteq V_2$ of size $q_2$ for which there exists a set $\tilde{B} \subseteq B$ of size at least $(1 - \varepsilon_1)q_2$ such that the pair $(V_1, \tilde{B})$ is $(\varepsilon_1)$-regular with directed density $d'$ for some $(1 - \varepsilon)d \leq d' \leq (1 + \varepsilon)d$. Clearly $|{\mathcal F}_1| \leq \binom{|V_2|}{q_2}$. Let ${\mathcal F}_2 = \binom{V_2}{q_2} \setminus {\mathcal F}_1$. Since $q_2 \geq C d^{-1} \geq C_2 d^{-1}$, it follows by Proposition~\ref{regularSmallSets} that $|{\mathcal F}_2| \leq \beta^{q_2} \binom{|V_2|}{q_2}$. Hence there are at most $\binom{|V_1|}{q_1} \beta^{q_2} \binom{|V_2|}{q_2}$ pairs $(A,B) \in \binom{V_1}{q_1} \times {\mathcal F}_2$ which do not satisfy Condition (ii) of Corollary~\ref{cor::regularityOfSmallPairs}.  

Fix some arbitrary $B \in {\mathcal F}_1$. Let $\tilde{B} \subseteq B$ be a set of size at least $(1 - \varepsilon_1)q_2$ such that the pair $(V_1, \tilde{B})$ is $(\varepsilon_1)$-regular with directed density $d'$ for some $(1 - \varepsilon)d \leq d' \leq (1 + \varepsilon)d$. Note that $|\tilde{B}| \geq (1 - \varepsilon')q_2$ since $\varepsilon_1 \leq \varepsilon'$. Since $q_1 \geq C d^{-1} \geq C_1 d^{-1}$, it follows by Proposition~\ref{regularSmallSets} that the number of sets $A \in \binom{V_1}{q_1}$ such that for every $\tilde{A} \subseteq A$ of size at least $(1 - \varepsilon')q_1$ the pair $(\tilde{A}, \tilde{B})$ is not $(\varepsilon')$-regular with directed density $d''$ for any $(1 - \varepsilon_1)d' \leq d'' \leq (1 + \varepsilon_1)d'$ is at most $\beta^{q_1} \binom{|V_1|}{q_1}$. Since $\varepsilon \leq \varepsilon_1 \leq \varepsilon'$ and $(1 - \varepsilon)d \leq d' \leq (1 + \varepsilon)d$ it follows that the number of sets $A \in \binom{V_1}{q_1}$ such that for every $\tilde{A} \subseteq A$ of size at least $(1 - \varepsilon')q_1$ the pair $(\tilde{A}, \tilde{B})$ is not $(\varepsilon')$-regular with directed density $d''$ for any $(1 - \varepsilon')^2 d \leq d'' \leq (1 + \varepsilon')^2 d$ is also at most $\beta^{q_1} \binom{|V_1|}{q_1}$. Multiplying this bound by the size of ${\mathcal F}_1$, it follows that there are at most $\beta^{q_1} \binom{|V_1|}{q_1} \binom{|V_2|}{q_2}$ pairs $(A,B) \in \binom{V_1}{q_1} \times {\mathcal F}_1$ which do not satisfy Condition (ii) of Corollary~\ref{cor::regularityOfSmallPairs}.       

We conclude that the number of pairs $(Q_1, Q_2)$ such that $Q_1 \subseteq V_1$ is of size $q_1$ and $Q_2 \subseteq V_2$ is of size $q_2$ which do not satisfy Condition (ii) of Corollary~\ref{cor::regularityOfSmallPairs} is at most
$$
\beta^{q_1} \binom{|V_1|}{q_1} \binom{|V_2|}{q_2} + \beta^{q_2} \binom{|V_1|}{q_1} \binom{|V_2|}{q_2} \,.
$$ 
\end{proof}

The following simple lemma is an immediate corollary of the definition of $(\varepsilon, p)$-regularity.
\begin{lemma} \label{smallDegreeVertices}
Let $(A,B)$ be an $(\varepsilon)$-regular pair with directed density $d$. Let $X \subseteq A$ and $Y \subseteq B$ be sets of size $|X| \geq \varepsilon |A|$ and $|Y| \geq \varepsilon |B|$. Then $\deg^+(x,Y) \leq (1 + \varepsilon)d|Y|$ for all but at most $\varepsilon |A|$ vertices $x \in A$ and $\deg^+(x,Y) \geq (1 - \varepsilon)d|Y|$ for all but at most $\varepsilon |A|$ vertices $x \in A$. Similarly, $\deg^-(y,X) \leq (1 + \varepsilon)d|X|$ for all but at most $\varepsilon |B|$ vertices $y \in B$ and $\deg^-(y,X) \geq (1 - \varepsilon)d|X|$ for all but at most $\varepsilon |B|$ vertices $y \in B$.
\end{lemma}

A (straightforward) proof for the analogous statement for undirected graphs can be found e.g.\ in~\cite{GS}. Since we only care about the arcs oriented from $A$ to $B$, 
Lemma~\ref{smallDegreeVertices} is essentially equivalent to Lemma 4.1 from~\cite{GS}.

\section{Properties of random directed graphs}
\label{sec::randomDigraphs}
In this section we collect several results about random directed graphs which we will need later in the proof of our main theorem.

\begin{lemma} \label{badSet} 
Let $n$ be a positive integer and let $\log n/n \ll p = p(n) \leq 1$. Let $\varepsilon > 0$ be arbitrarily small, let $c > 0$ be a constant, and let $D=(V,E) \in {\mathcal D}(n,p)$. For a set $Y \subseteq V$, let $B_Y$ denote the set of all vertices $u \in V \setminus Y$ for which $|\deg^+_D(u,Y) - |Y| p| \geq \varepsilon |Y| p$ or $|\deg^-_D(u,Y) - |Y| p| \geq \varepsilon |Y| p$. Let $b = \max \{|B_Y| : Y \subseteq V, \; |Y| \geq c n\}$, then a.a.s. $b \leq p^{-1} \log n$. 
\end{lemma}

\begin{proof}
Let $u \in V$ be any vertex and let $Y \subseteq V \setminus \{u\}$ be any set of size at least $c n$. Note that $\deg^+_D(u,Y) \sim Bin(|Y|, p)$ and similarly $\deg^-_D(u,Y) \sim Bin(|Y|, p)$. Thus $\mathbb{E}(\deg^+_D(u,Y)) = \mathbb{E}(\deg^-_D(u,Y)) = |Y| p$. It follows by Theorem~\ref{th::Chernoff} (iii) that $Pr \left(|\deg^+_D(u,Y) - |Y| p| \geq \varepsilon |Y| p \right)$ is bounded from above by $2 \exp \left\{- \frac{\varepsilon^2}{3} \cdot |Y| p \right\} \leq e^{- \varepsilon^2 c' n p}$, where $c' > 0$ is an appropriate constant, and the same holds for $Pr \left(|\deg^-_D(u,Y) - |Y| p| \geq \varepsilon |Y| p \right)$ as well. Hence, 
\begin{eqnarray*} 
Pr(b \geq p^{-1} \log n) &=& Pr(\exists Y \subseteq V \text{ of size at least } c n \text{ such that } |B_Y| \geq 
p^{-1} \log n)\\ &\leq& 2^n \binom{n}{p^{-1} \log n} \left(2 e^{- \varepsilon^2 c' n p} \right)^{p^{-1} \log n} = o(1) \,. 
\end{eqnarray*}
\end{proof}

\begin{lemma} \label{lem::degreeTooHigh} 
Let $0 < c \leq 1$ be a constant, let $n$ be a positive integer, let $\log n/\sqrt{n} \ll p = p(n) \leq 1$ and let $D=(V,E) \in {\mathcal D}(n,p)$. Let $\ell = \ell(n) \leq n$ be an integer satisfying $\ell p^2 \gg \log n$. Then a.a.s. 
\begin{description}
\item [(i)] For every $A \subseteq V$ of size $c \ell p \leq a \leq 2 \ell p$ we have $|\{u \in V \setminus A : \deg_D^+(u, A) \geq 2 p a\}| \leq \ell p$.
\item [(ii)] For every $A \subseteq V$ of size $\ell p^{3/2} \leq a \leq 2 \ell p$ we have $|\{u \in V \setminus A : \deg_D^+(u, A) \geq 7 \sqrt{p} a\}| \leq \ell p^{3/2}$.
\end{description} 
\end{lemma}

\begin{proof}
Starting with (i), let $F \subseteq V \setminus A$ be an arbitrary set of size $\ell p$. Clearly $e_D(F, A) \sim Bin(|F||A|, p)$ and thus $\mathbb{E}(e_D(F, A)) = |F||A|p = \ell p^2 a$. It follows by Theorem~\ref{th::Chernoff} (ii) that $Pr(e_D(F, A) \geq 2\ell p^2 a) \leq e^{- \ell p^2 a/3}$. We conclude that the probability that there exists a set $A$ of size $c \ell p \leq a \leq 2 \ell p$ and a disjoint set $F$ of size $\ell p$ such that $\deg_D^+(u, A) \geq 2 p a$ for every $u \in F$ is at most 
\begin{eqnarray*}
\sum_{a = c \ell p}^{2 \ell p} \binom{n}{a} \binom{n-a}{\ell p} Pr(e_D(F, A) \geq 2\ell p^2 a)
&\leq& \sum_{a = c \ell p}^{2 \ell p} n^a n^{\ell p} e^{- \ell p^2 a/3}\\
&\leq& 2 \ell p \cdot \exp \left\{3 \ell p \log n - c \ell^2 p^3/3 \right\}\\
&=& o(1) \,,
\end{eqnarray*} 
where the last equality follows by our assumption that $\ell p^2 \gg \log n$.

Similarly for (ii), let $F \subseteq V \setminus A$ be an arbitrary set of size $\ell p^{3/2}$. Clearly $e_D(F, A) \sim Bin(|F||A|, p)$ and thus $\mathbb{E}(e_D(F, A)) = \ell p^{5/2} a$. It follows by Theorem~\ref{th::Chernoff} (iv) that $Pr(e_D(F, A) \geq \ell p^{3/2} \cdot 7 \sqrt{p} a) \leq e^{- 7 \ell p^2 a}$. We conclude that the probability that there exists a set $A$ of size $\ell p^{3/2} \leq a \leq 2 \ell p$ and a disjoint set $F$ of size $\ell p^{3/2}$ such that $\deg_D^+(u, A) \geq 7 \sqrt{p} a$ for every $u \in F$ is at most 
\begin{eqnarray*}
\sum_{a = \ell p^{3/2}}^{2 \ell p} \binom{n}{a} \binom{n-a}{\ell p^{3/2}} Pr(e_D(F, A) \geq 7 \ell p^2 a)
&\leq& \sum_{a = \ell p^{3/2}}^{2 \ell p} n^a n^{\ell p^{3/2}} e^{- 7 \ell p^2 a}\\
&\leq& \sum_{a = \ell p^{3/2}}^{2 \ell p} \exp \left\{2 a \log n - 7 \ell p^2 a \right\}\\
&=& o(1) \,,
\end{eqnarray*} 
where the second inequality holds since $a \geq \ell p^{3/2}$ and the last equality follows by our assumption that $\ell p^2 \gg \log n$.
\end{proof}

The following two lemmas will be useful in Stage 4 of the proof of the main result, where we will want to extend a long cycle to a Hamilton cycle. 

\begin{lemma} \label{lem::edgeDistribution}
Let $\alpha$ and $\beta$ be positive real numbers satisfying $3 \alpha < \beta$. Let $n$ be a positive integer, let $\sqrt{\log n/n} \ll p = p(n) \leq 1$ and let $D \in {\mathcal D}(n,p)$. Then a.a.s. the following holds for every $\emptyset \neq S \subseteq [n]$ of size $s \leq \alpha n$ and every set $T$ of $s$ arcs with both endpoints in $[n] \setminus S$ which span a digraph with maximum out-degree one and maximum in-degree one: there are less than $\beta s p^2 n$ pairs $((x,y),z) \in T \times S$ such that $(x,z) \in E(D)$ and $(z,y) \in E(D)$. 
\end{lemma}

\begin{proof}
Fix sets $S$ and $T$ as in the statement of the lemma. Given a vertex $z \in S$ and an arc $e = (x,y) \in T$, let $A_e^z$ denote the event: ``$(x,z) \in E(D)$ and $(z,y) \in E(D)$''. Note that
\begin{enumerate}
\item $Pr(A_e^z) = p^2$ holds for every $z \in S$ and every arc $e \in T$. 
\item If $u \neq v$ are vertices in $S$ and $e, f$ are not necessarily distinct arcs in $T$, then the events $A_e^u$ and $A_f^v$ are independent.
\item If $z \in S$ and $e, f \in T$ are two disjoint or anti-parallel arcs, then the events $A_e^z$ and $A_f^z$ are independent. 
\item If $z \in S$ and $e, f \in T$ share one vertex $v$, then $e \cup f$ is a directed path of length 2 (otherwise $v$ will have out-degree or in-degree at least 2). Hence, the events $A_e^z$ and $A_f^z$ are independent.  
\end{enumerate} 
It readily follows from the above four properties that for every $B \subseteq T \times S$, the events of $\{A_e^z : (e, z) \in B\}$ are mutually independent and thus $Pr \left(\bigwedge_{(e, z) \in B} A_e^z \right) = (p^2)^{|B|}$. 

We thus conclude that the probability that there exist sets $S$ and $T$ as in the assertion of the lemma for which there are at least $\beta s p^2 n$ pairs $((x,y),z) \in T \times S$ such that $(x,z) \in E(D)$ and $(z,y) \in E(D)$ is at most
\begin{eqnarray*}
\sum_{s=1}^{\alpha n} \binom{n}{s} \binom{(n-s)^2}{s} \binom{s^2}{\beta s p^2 n} (p^2)^{\beta s p^2 n} 
\leq \sum_{s=1}^{\alpha n} n^s n^{2s} \left(\frac{e s}{\beta n}\right)^{\beta s p^2 n}
\leq \sum_{s=1}^{\alpha n} \left[n^3 \left(\frac{e}{3}\right)^{\beta p^2 n}\right]^s
= o(1) \,,
\end{eqnarray*}
where the second inequality follows since $s \leq \alpha n < \beta n/3$ and the last equality follows by the assumed lower bound on $p$.
\end{proof}

\begin{lemma} \label{lem::matchingStage4}
Let $G = (A \cup B, E)$ be a bipartite graph on $n = |A| + |B|$ vertices, where $|A| \leq |B|$, and let $\delta = \delta(n)$ be a positive integer. If $G$ satisfies the following two properties
\begin{description}
\item [(i)] $\deg_G(x) \geq \delta$ holds for every $x \in A$;
\item [(ii)] $e_G(X,Y) < \delta |X|$ holds for every $X \subseteq A$ and every $Y \subseteq B$ such that $|X| = |Y|$;
\end{description}
then there exists a matching of $G$ which saturates $A$.
\end{lemma}

\begin{proof}
In order to prove the existence of such a matching, we will use Hall's Theorem (see e.g.~\cite{West}), that is, we will prove that $|N_G(S)| \geq |S|$ holds for every $S \subseteq A$, where $N_G(S) := \{v \in B : \exists u \in S \textrm{ such that } \{u,v\} \in E\}$. Suppose for a contradiction that there exists a set $S \subseteq A$ of size $s$ such that $|N_G(S)| < |S|$; clearly $S \neq \emptyset$. Let $N_G(S) \subseteq T \subseteq B$ be a set of size $s$. Since $\deg_G(x) \geq \delta$ holds for every $x \in A$ by Property (i) above, it follows that $e_G(S,T) \geq \delta s$. On the other hand $e_G(S,T) < \delta s$ holds by Property (ii) above. Clearly, this is a contradiction.  
\end{proof}

A central part of our proof of the lower bound in Theorem~\ref{directedHam} will consist of building a long directed path with certain properties. The use of the Sparse Diregularity Lemma will result in certain problematic vertices. In the course of building this path we will try to avoid some of these problematic vertices while making sure we include others. In what follows we describe two kinds of problematic vertices we will have to deal with. Since we apply the Sparse Diregularity Lemma to a subdigraph of ${\mathcal D}(n,p)$, we expect to encounter a relatively small number of problematic vertices. This will be made precise in Lemmas~\ref{badConfigurationsI} and~\ref{badConfigurationsII}.

Let $\varepsilon, \varepsilon'$ be positive real numbers. Let $D=(V,E)$ be a digraph on $n$ vertices, let $q_1 = q_1(n)$, $q_2 = q_2(n)$ and $\ell = \ell(n)$ be positive integers, and let $0 < d = d(n) \leq 1$. Let $X$ and $Y$ be disjoint subsets of $V$ of size $\ell$ each, such that the pair $(X,Y)$ is $(\varepsilon)$-regular with directed density $d$.

A vertex $u \in V$ is called \emph{bad of type I} (with respect to $D$, $\ell$, $X$, $Y$, $\varepsilon$, $\varepsilon'$, $d$, $q_1$, and $q_2$) if $u \notin X \cup Y$ and at least one of the following conditions holds 
\begin{description}
\item [(I.1)] There exists a set $Q \subseteq N_D^+(u, X)$ of size $q_1 \leq |Q| \leq q_2$ such that for every $\tilde{Q} \subseteq Q$ of size at least $(1 - \varepsilon')|Q|$ the pair $(\tilde{Q}, Y)$ is \emph{not} $(\varepsilon')$-regular with directed density $d'$ for any $(1 - \varepsilon)d \leq d' \leq (1 + \varepsilon)d$. 

\item [(I.2)] There exists a set $Q \subseteq N_D^-(u, Y)$ of size $q_1 \leq |Q| \leq q_2$ such that for every $\tilde{Q} \subseteq Q$ of size at least $(1 - \varepsilon')|Q|$ the pair $(X, \tilde{Q})$ is \emph{not} $(\varepsilon')$-regular with directed density $d'$ for any $(1 - \varepsilon)d \leq d' \leq (1 + \varepsilon)d$. 
\end{description}

Given sets $X$ and $Y$ as above, let $T_1 \subseteq V$ denote a set of bad vertices of type I (with respect to these specific $X$ and $Y$). The digraph $D$, the sets $X, Y$ and the set $T_1$ are said to form a \emph{$(|T_1|, \ell, \varepsilon, \varepsilon', d, q_1, q_2)$ bad configuration of type I}. We will prove that a.a.s. no subdigraph of ${\mathcal D}(n,p)$ contains such a configuration with a large set $T_1$. 

\begin{lemma} \label{badConfigurationsI} Let $n$ be a positive integer, $n^{-1/2} \ll p = p(n) \leq 1$, and $D \in {\mathcal 
D}(n,p)$. Let $0 < \varepsilon', \rho, \lambda, \xi < 1$ be constants and let $d = \xi p$. Then, there exists $0 < \varepsilon = \varepsilon(\xi, \lambda, \rho, \varepsilon') \leq \varepsilon'$ such that, for every positive integer $n^{3/4} < \ell = \ell(n) < n$ satisfying $\ell \gg p^{-2}$, a.a.s. there are no $(\rho \ell, \ell, \varepsilon, \varepsilon', d, \lambda \ell p, 2 \ell p)$ bad configurations of type I in any subdigraph of $D$. 
\end{lemma}

\begin{proof}
Since $\rho, \lambda, \xi$ are positive constants, by choosing $\beta > 0$ to be sufficiently small we can guarantee that
$\left(\frac{e}{\xi}\right)^{\xi} \left(\frac{\beta e}{\lambda}\right)^{\lambda \rho} \leq 1/4$. Let $\varepsilon_0 = 
\varepsilon_0(\beta, \varepsilon') > 0$ and $C = C(\varepsilon')$ be the constants whose existence follows from Proposition~\ref{regularSmallSets} and let $\varepsilon = \min \{\varepsilon_0, \varepsilon'\}$.

We would like to bound from above the expected number of large bad configurations of type I in any subdigraph of $D$. Fix 
two disjoint sets $X, Y \subseteq [n]$ of size $\ell$ each, and a set $T_1 \subseteq [n] \setminus (X \cup Y)$ of size $\rho \ell$ for
some $n^{3/4} < \ell < n$ such that $\ell \gg p^{-2}$. The number of choices of $X, Y$ and $T_1$ is at most
$\binom{n}{\ell}^2 \binom{n}{\rho \ell} \leq 2^{3n}$. Choose also $d \ell^2$ arcs directed from $X$ to $Y$. The number of ways to choose these arcs is $\binom{\ell^2}{d \ell^2}$ and each such arc set appears in a subdigraph of our random digraph $D$ with probability at most $p^{d \ell^2}$.

Since $C$ and $\lambda$ are constants, $\ell \gg p^{-2}$ and $d = \Theta(p)$ we have that $\lambda \ell p = \omega(p^{-1}) > C d^{-1}$. Therefore, if the pair $(X,Y)$ is $(\varepsilon)$-regular, then from Proposition~\ref{regularSmallSets} it follows that the number of sets $Q \subseteq X$ of size $q$ for some fixed $\lambda \ell p \leq q \leq 2 \ell p$ such that for every $\tilde{Q} \subseteq Q$ of size at least $(1 - \varepsilon')q$ the pair $(\tilde{Q}, Y)$ is \emph{not} $(\varepsilon')$-regular with directed density $d'$ for any $(1 - \varepsilon)d \leq d' \leq (1 + \varepsilon)d$ is at most $\beta^q \binom{\ell}{q}$. Let $u \in T_1$ be an arbitrary vertex. Expose all arcs of $D$ with one endpoint in $\{u\}$ and the other in $X \cup Y$. For any subdigraph $D'$ of $D$ and every $Q \subseteq X$, the probability that $Q \subseteq N_{D'}^+(u, X)$ is bounded from above by the probability that $Q \subseteq N_D^+(u, X)$ which is $p^{|Q|}$. It follows by a union bound argument that the probability that (I.1) holds for $u$ is at most $\sum_{q = \lambda \ell p}^{2 \ell p} \beta^q \binom{\ell}{q} p^q$ (note that if $\deg_{D'}^+(u, X) < \lambda \ell p$, then the probability that (I.1) holds for $u$ is 0). An analogous argument shows that exactly the same bound applies to the probability that (I.2) holds for $u$. Hence, the probability that $u$ is bad of type I with respect to these $\ell$, $X$ and $Y$ is at most $2 \sum_{q = \lambda \ell p}^{2 \ell p} \beta^q \binom{\ell}{q} p^q$. Since for distinct $u, u' \in T_1$, the validity of (I.1) (and similarly (I.2)) involves disjoint sets of edges, it follows that the events ``$u$ is bad of type I'' and ``$u'$ is bad of type I'' are independent (all events are with respect to the fixed $X$ and $Y$ and given the arcs between $X$ and $Y$). Hence, the probability that for given $X$, $Y$ and arcs between them, all the vertices of $T_1$ are bad of type I is at most

\begin{eqnarray} \label{eq::badVertices}
\left(2 \sum_{q = \lambda \ell p}^{2 \ell p} \beta^q \binom{\ell}{q} p^q\right)^{\rho \ell} \leq
\left(2 \sum_{q = \lambda \ell p}^{2 \ell p} \left(\beta \cdot \frac{e \ell}{q} \cdot p\right)^q \right)^{\rho \ell}
\leq \left(4 \ell p \left(\beta \cdot \frac{e}{\lambda} \right)^{\lambda \ell p} \right)^{\rho \ell} \leq \left(2 \left(\frac{\beta e}{\lambda}\right)^{\lambda \rho} \right)^{\ell^2 p}
\end{eqnarray}
where the second inequality follows since $\left(\beta e \ell p/q \right)^q$ is maximized by the smallest value of $q$ in the given range and in the last inequality we used the bounds $(4 \ell p)^{\rho} \leq 4 \ell p \leq 2^{\ell p}$ which hold since $\rho \leq 1$ and $\ell p \gg 1$. 

Summing over all appropriate choices of $\ell, X, Y, T_1$ and $d \ell^2$ arcs, directed from $X$ to $Y$, and using the 
estimate~\eqref{eq::badVertices}, by linearity of expectation, we conclude that the expected number of $(\rho \ell, \ell, \varepsilon, \varepsilon', d, \lambda \ell p, 2 \ell p)$ bad configurations of type I in some subdigraph of $D$ is at most

\begin{eqnarray*}
\sum_{\ell = n^{3/4}}^{n} 2^{3n} \binom{\ell^2}{d \ell^2} p^{d \ell^2} \left(2 \left(\frac{\beta 
e}{\lambda}\right)^{\lambda \rho} \right)^{\ell^2 p} 
&\leq& 2^{3 n} \sum_{\ell = n^{3/4}}^{n} \left(\frac{ep}{d}\right)^{d \ell^2} \left(2 \left(\frac{\beta e}{\lambda}\right)^{\lambda \rho} \right)^{\ell^2 p}\\
&=& 2^{3 n} \sum_{\ell = n^{3/4}}^{n} \left(2\left(\frac{e}{\xi}\right)^{\xi} \left(\frac{\beta e}{\lambda}\right)^{\lambda \rho} \right)^{\ell^2 p} \\
&\leq& 2^{3 n} \sum_{\ell = n^{3/4}}^{n} 2^{-\ell^2 p}
\leq n 2^{3 n} 2^{- n^{3/2} p} = o(1) \,,
\end{eqnarray*}
where the first equality follows since $d = \xi p$, the second inequality holds by our choice of $\beta$ and the last equality holds since $p \gg n^{-1/2}$. Hence, by Markov's inequality there are a.a.s. no $(\rho \ell, \ell, \varepsilon, \varepsilon', d, \lambda \ell p, 2 \ell p)$ bad configurations of type I in any subdigraph of $D$. 
\end{proof}

Let $\varepsilon, \varepsilon'$ be positive real numbers. Let $D=(V,E)$ be a digraph on $n$ vertices, let $q_1 = q_1(n)$, $q_2 = q_2(n)$, $\ell = \ell(n)$ and $r \leq n/\ell$ be positive integers, and let $0 < d = d(n) \leq 1$. Let $V_1, \ldots, V_r$ be pairwise disjoint subsets of $V$, of size $\ell$ each, such that, for every $1 \leq i \leq r$, the pair $(V_i, V_{i+1})$ is $(\varepsilon)$-regular with directed density $d$ (throughout this section $V_{r+1}$ should be read as $V_1$).

For $1 \leq i \leq r$, a vertex $u \in V$ is called \emph{$i$-bad of type II} (with respect to $D$, $\ell$, $V_1, \ldots, V_r$, $\varepsilon$, $\varepsilon'$, $d$, $q_1$, and $q_2$) if at least one of the following conditions holds

\begin{description}
\item [$(i)$] $u$ is bad of type I with respect to $D$, $\ell$, $V_{i-1}$, $V_i$, $\varepsilon$, $\varepsilon'$, $d$, $q_1$, and $q_2$.
\item [$(ii)$] $u$ is bad of type I with respect to $D$, $\ell$, $V_i$, $V_{i+1}$, $\varepsilon$, $\varepsilon'$, $d$, $q_1$, and $q_2$.
\item [$(iii)$] $u$ is bad of type I with respect to $D$, $\ell$, $V_{i+1}$, $V_{i+2}$, $\varepsilon$, $\varepsilon'$, $d$, $q_1$, and $q_2$.
\item [$(iv)$] $u \notin V_i \cup V_{i+1}$ and there exist sets $Q_i \subseteq N_D^-(u, V_i)$ and $Q_{i+1} \subseteq N_D^+(u, V_{i+1})$ of sizes $q_1 \leq |Q_i|, |Q_{i+1}| \leq q_2$ such that, for every $\tilde{Q}_i \subseteq Q_i$ and every $\tilde{Q}_{i+1} \subseteq Q_{i+1}$ such that $|\tilde{Q}_i| \geq (1 - \varepsilon')|Q_i|$ and $|\tilde{Q}_{i+1}| \geq (1 - \varepsilon')|Q_{i+1}|$, the pair $(\tilde{Q}_i, \tilde{Q}_{i+1})$ is \emph{not} $(\varepsilon')$-regular with directed density $d''$ for any $(1 - \varepsilon')^2 d \leq d'' \leq (1 + \varepsilon')^2 d$.
\end{description}  

Let $\alpha > 0$ be a constant. A vertex $u \in V$ is called \emph{bad of type II} (with respect to $D$, $\ell$, $V_1, \ldots, V_r$, $\alpha$, $\varepsilon$, $\varepsilon'$, $d$, $q_1$, and $q_2$) if there exists a set $I_u \subseteq [r]$ of size $|I_u| \geq \alpha r/40$ such that, $u$ is $i$-bad of type II for every $i \in I_u$.

Given sets $V_1, \ldots, V_r$ as above, let $T_2 \subseteq V$ denote the set of bad vertices of type II (with respect to these specific $V_1, \ldots, V_r$). The digraph $D$, the sets $V_1, \ldots, V_r$, and the set $T_2$ are said to form a \emph{$(|T_2|, r, \ell, \alpha, \varepsilon, \varepsilon', d, q_1, q_2)$ bad configuration of type II}. We will prove that a.a.s. no subdigraph of ${\mathcal D}(n,p)$ contains such a configuration with a large set $T_2$.

\begin{lemma} \label{badConfigurationsII}
Let $r \geq 2$ be an integer, let $n$ be a positive integer, $n^{-1/2} \ll p = p(n) \leq 1$, and $D \in {\mathcal D}(n,p)$. Let $0 < \alpha, \varepsilon', \rho, \lambda, \xi < 1$ be constants and let $d = \xi p$. Then, there exists $0 < \varepsilon = \varepsilon(\xi, \lambda, \rho, \varepsilon') \leq \varepsilon'$ such that, for every positive integer $n^{3/4} < \ell = \ell(n) \leq n/r$ satisfying $\ell \gg p^{-2}$, a.a.s. there are no $(\rho \ell, r, \ell, \alpha, \varepsilon, \varepsilon', d, \lambda \ell p, 2 \ell p)$ bad configurations of type II in any subdigraph of $D$.
\end{lemma}

\begin{proof}
Since $\rho, \lambda, \xi$ are positive constants, by choosing $\beta > 0$ to be sufficiently small we can guarantee that
$\left(\frac{e}{\xi}\right)^{\xi} (2 e^4 \beta^{\lambda})^{\alpha \rho/640} \leq 1/2$. Let $0 < \varepsilon_0^1 = \varepsilon_0^1(\xi, \lambda, \alpha \rho/640, \varepsilon') \leq \varepsilon'$ be the constant whose existence follows from Lemma~\ref{badConfigurationsI}. Let $\varepsilon_0^2 = \varepsilon_0^2(\beta, \varepsilon') > 0$ and $C = C(\beta, \varepsilon')$ be the constants whose existence follows from Corollary~\ref{cor::regularityOfSmallPairs}. Let $\varepsilon = \min \{\varepsilon_0^1, \varepsilon_0^2\}$.

We would like to bound from above the expected number of large bad configurations of type II in any subdigraph of $D$. Assume then that $D, V_1, \ldots, V_r$ and $T_2$ form a $(\rho \ell, r, \ell, \alpha, \varepsilon, \varepsilon', d, \lambda \ell p, 2 \ell p)$ bad configuration of type II. Note that by definition, for every vertex $v \in T_2$ one of the conditions $(i)-(iv)$ holds for at least $(\alpha r/40)/4 = \alpha r/160$ indices. This implies that $T_2$ contains a subset $S$ of size at least $\rho \ell/4$ such that for all vertices in $S$ the same condition holds for at least $\alpha r/160$ indices.

Given a subdigraph $D'$ of $D$, assume first that there exists a set $S \subseteq T_2$ of size $|S| \geq \rho \ell/4$ such that for every $x \in S$, condition $(i)$ holds for at least $\alpha r/160$ indices $i \in I_x$ (with respect to these $V_1, \ldots, V_r$). It follows by averaging that there must exist some $1 \leq j \leq r$ and a set $S' \subseteq S$ of size $|S'| \geq (\rho \ell/4 \cdot \alpha r/160)/r = \alpha \rho \ell/640$ such that condition $(i)$ holds for $j$ and for every $y \in S'$. Therefore, the digraph $D'$, the sets $V_{j-1}, V_j$ and the set $S'$ form an $(\alpha \rho \ell/640, \ell, \varepsilon, \varepsilon', d, \lambda \ell p, 2 \ell p)$ bad configuration of type I. However, by our choice of $\varepsilon$ and by Lemma~\ref{badConfigurationsI} the probability of this happening is $o(1)$. Using an analogous argument for conditions $(ii)$ and $(iii)$, we conclude that it suffices to prove that the probability that there exists a set $S \subseteq T_2$ of size $|S| \geq \rho \ell/4$ such that for every $x \in S$, condition $(iv)$ holds for at least $\alpha r/160$ indices $i \in I_x$ is $o(1)$. Let $S$ be such a set. It again follows by averaging that there must exist some $1 \leq j \leq r$ and a set $S' \subseteq S$ of size $|S'| \geq (\rho \ell/4 \cdot \alpha r/160)/r = \alpha \rho \ell/640$ such that condition $(iv)$ holds for $j$ and for every $y \in S'$. It thus suffices to prove that the probability that there exist pairwise disjoint vertex sets $X,Y$ and $B$, where $|X| = |Y| = \ell$ and $|B| = \alpha \rho \ell/640$ such that the pair $(X,Y)$ is $(\varepsilon)$-regular with directed density $d$ and condition $(iv)$ holds for every $u \in B$ with respect to $X$ and $Y$ is $o(1)$. As shown below, this can be done similarly to the proof of Lemma~\ref{badConfigurationsI}.

There are at most $\binom{n}{\ell}^2 \binom{n}{\alpha \rho \ell/640} \leq 2^{3n}$ ways to choose $X,Y$ and $B$. The number of ways to choose $d \ell^2$ arcs, directed from $X$ to $Y$, is $\binom{\ell^2}{d \ell^2}$ and each such arc set appears in a subdigraph of $D$ with probability at most $p^{d \ell^2}$.

Since $C$ and $\lambda$ are constants, $\ell \gg p^{-2}$ and $d = \Theta(p)$ we have that $\lambda \ell p = \omega(p^{-1}) > C d^{-1}$. Therefore, if the pair $(X,Y)$ is $(\varepsilon)$-regular, then from Corollary~\ref{cor::regularityOfSmallPairs} it follows that, for any fixed integers $\lambda \ell p \leq q_X, q_Y \leq 2 \ell p$, there are at most $(\beta^{q_X} + \beta^{q_Y}) \binom{|X|}{q_X} \binom{|Y|}{q_Y}$ pairs $(Q_X, Q_Y)$ such that $Q_X \subseteq X$ is of size $q_X$ and $Q_Y \subseteq Y$ is of size $q_Y$ and, moreover, for every $\tilde{Q}_X \subseteq Q_X$ and every $\tilde{Q}_Y \subseteq Q_Y$ such that $|\tilde{Q}_X| \geq (1 - \varepsilon') q_X$ and $|\tilde{Q}_Y| \geq (1 - \varepsilon') q_Y$, the pair $(\tilde{Q}_X, \tilde{Q}_Y)$ is \emph{not} $(\varepsilon')$-regular with directed density $d''$ for any $(1 - \varepsilon')^2 d \leq d'' \leq (1 + \varepsilon')^2 d$. Let $u \in B$ be an arbitrary vertex. Expose all arcs of $D$ with one endpoint in $\{u\}$ and the other in $X \cup Y$. For any subdigraph $D'$ of $D$ and every $Q_X \subseteq X$ and $Q_Y \subseteq Y$, the probability that $Q_X \subseteq N_{D'}^-(u,X)$ and $Q_Y \subseteq N_{D'}^+(u,Y)$ is at most $p^{|Q_X|+|Q_Y|}$. Hence, the probability that condition $(iv)$ holds for $u$ with respect to $X$ and $Y$ is at most   
\begin{eqnarray} \label{eq::iv}
&& \sum_{q_X = \lambda \ell p}^{2 \ell p} \sum_{q_Y = \lambda \ell p}^{2 \ell p} (\beta^{q_X} + \beta^{q_Y}) \binom{\ell}{q_X} \binom{\ell}{q_Y} p^{q_X + q_Y} 
\leq 2 \beta^{\lambda \ell p} \sum_{q_X = \lambda \ell p}^{2 \ell p} \sum_{q_Y = \lambda \ell p}^{2 \ell p} \left(\frac{e \ell p}{q_X} \right)^{q_X} \left(\frac{e \ell p}{q_Y} \right)^{q_Y} \nonumber \\ 
&\leq& 2 \beta^{\lambda \ell p} (2 \ell p)^2 e^{4 \ell p} = 8 (\ell p)^2 (e^4 \beta^{\lambda})^{\ell p} \leq (2 e^4 \beta^{\lambda})^{\ell p} \,,
\end{eqnarray}
where the first inequality holds since $f(q) := \beta^q$ is decreasing in the range $\lambda \ell p \leq q \leq 2 \ell p$ as $\beta < 1$, the second inequality holds since $h(q) := \left(\frac{e \ell p}{q} \right)^q$ is increasing in the range $\lambda \ell p \leq q \leq \ell p$ and in the last inequality we used the bound $8(\ell p)^2 \leq 2^{\ell p}$ which holds since $\ell p \gg 1$.

Let $u, u' \in B$ be any two vertices. Since $B$ is disjoint from $X \cup Y$, the validity of $(iv)$ for $u$ and for $u'$ (both with respect to $X$ and $Y$ and given the arcs between $X$ and $Y$) involves disjoint sets of edges. Hence, the events ``$(iv)$ holds for $u$ with respect to $X$ and $Y$'' and ``$(iv)$ holds for $u'$ with respect to $X$ and $Y$'' are independent. Thus, using~\eqref{eq::iv} we conclude that the probability that condition $(iv)$ holds for every $x \in B$ is at most
\begin{equation} \label{eq::manyVerticesManyIndices}
\left((2 e^4 \beta^{\lambda})^{\ell p} \right)^{\alpha \rho \ell/640} = (2 e^4 \beta^{\lambda})^{\alpha \rho \ell^2 p/640}
\end{equation}

Summing over all appropriate choices of $\ell, X, Y, B$ and $d \ell^2$ arcs, directed from $X$ to $Y$, and using estimate~\eqref{eq::manyVerticesManyIndices}, by linearity of expectation, we conclude that the expected number of $(\rho \ell, r, \ell, \alpha, \varepsilon, \varepsilon', d, \lambda \ell p, 2 \ell p)$ bad configurations of type II in some subdigraph of $D$ is at most

\begin{eqnarray*}
\sum_{\ell = n^{3/4}}^{n} 2^{3 n} \binom{\ell^2}{d \ell^2} p^{d \ell^2} (2 e^4 \beta^{\lambda})^{\alpha \rho \ell^2 p/640} 
&\leq& 2^{3 n} \sum_{\ell = n^{3/4}}^{n} \left(\frac{ep}{d}\right)^{d \ell^2} (2 e^4 \beta^{\lambda})^{\alpha \rho \ell^2 p/640}\\
&=& 2^{3 n} \sum_{\ell = n^{3/4}}^{n} \left(\left(\frac{e}{\xi}\right)^{\xi} (2 e^4 \beta^{\lambda})^{\alpha \rho/640} \right)^{\ell^2 p}\\
&\leq& n 2^{3 n} 2^{- n^{3/2} p} = o(1) \,,
\end{eqnarray*}
where the first equality follows since $d = \xi p$, the second inequality holds by our choice of $\beta$ and in the last equality we used the assumed lower bound on $p$. Hence, by Markov's inequality there are a.a.s. no $(\rho \ell, r, \ell, \alpha, \varepsilon, \varepsilon', d, \lambda \ell p, 2 \ell p)$ bad configurations of type II in any subdigraph of $D$.
\end{proof}

The following lemma will be useful in the next section when we will show how to build a long cycle which can be used to absorb the remaining vertices so as to create a Hamilton cycle. 

\begin{lemma} \label{lem::manyGoodPairs}
Let $\alpha > 0$ be a constant, let $n$ be a positive integer and let $\log n/ \sqrt{n} \ll p = p(n) \leq 1$. Let $D' = (V,E)$ be a digraph obtained from $D \in {\mathcal D}(n,p)$ by deleting at most $(1/2 - \alpha) \deg_D^+(u)$ out-going arcs and at most $(1/2 - \alpha) \deg_D^-(u)$ in-going arcs at every vertex $u \in V(D)$. Let $\ell$, $V_1, \ldots, V_r$, $\varepsilon$, $\varepsilon'$, $d$, $q_1$, and $q_2$ be as in the definition of bad vertices of type II and assume further that $r \ell (1 - \alpha/4 - 2/r) (1 - \varepsilon - \alpha/2) \geq (1 - \alpha) n$. Let $U$ denote the set of all vertices $u \in V$ which satisfy the following two properties:
\begin{description}
\item [(a)] $u$ is not bad of type II (with respect to $D'$, $\ell$, $V_1, \ldots, V_r$, $\varepsilon$, $\varepsilon'$, $d$, $q_1$, and $q_2$).
\item [(b)] $\deg^+_D(u, V_i) \geq (1 - \varepsilon) \ell p$ and $\deg^-_D(u, V_i) \geq (1 - \varepsilon) \ell p$ for every $1 \leq i \leq r$. 
\end{description}
Then a.a.s. for every $u \in U$ there exists a set $I_u \subseteq [r]$ such that all of the following properties hold:
\begin{description}
\item [(i)] $|I_u| \geq \alpha r/40$.
\item [(ii)] $u \notin \bigcup_{i \in I_u} (V_i \cup V_{i+1})$.
\item [(iii)] $\deg^-_{D'}(u,V_i) \geq \alpha \ell p/2$ and $\deg^+_{D'}(u,V_{i+1}) \geq \alpha \ell p/2$ for every $i \in I_u$.
\item [(iv)] $(j - i) \mod r \geq 5$ and $(i - j) \mod r \geq 5$ for every $i \neq j \in I_u$.
\item [(v)] $u$ is not $i$-bad of type II for any $i \in I_u$ (with respect to $D'$, $\ell$, $V_1, \ldots, V_r$, $\varepsilon$, $\varepsilon'$, $d$, $q_1$, and $q_2$).
\end{description}
\end{lemma}

\begin{proof}
Asymptotically almost surely $(1 - o(1)) n p \leq \deg_D^+(u) \leq (1 + o(1)) n p$ and $(1 - o(1)) n p \leq \deg_D^-(u) \leq (1 + o(1)) n p$ hold for every $u \in V$. We will thus assume this throughout the proof. The remainder of the proof is deterministic. Fix an arbitrary vertex $u \in U$. Let $I_u^{(ii)} = \{1 \leq i \leq r : u \notin V_i \cup V_{i+1}\}$; clearly $|I_u^{(ii)}| \geq r-2$. Let $I_u^{(iii)}$ denote the set of indices of $I_u^{(ii)}$ which satisfy Property (iii) above; we claim that $|I_u^{(iii)}| \geq \alpha r/4$. Indeed, suppose for a contradiction that $|I_u^{(iii)}| < \alpha r/4$. Fix some $i \in I_u^{(ii)}$ for which Property (iii) is not satisfied. Since $u \in U$, it follows by Property (b) above that $|\{(v,u) \in E(D) \setminus E(D') : v \in V_i\}| \geq (1 - \varepsilon - \alpha/2) \ell p$ or $|\{(u,v) \in E(D) \setminus E(D') : v \in V_{i+1}\}| \geq (1 - \varepsilon - \alpha/2) \ell p$. Hence, in order to obtain $D'$ from $D$, one has to delete at least 
$$
(r - \alpha r/4 - 2) \cdot (1 - \varepsilon - \alpha/2) \ell p \geq (1 - \alpha) n p > (1/2 - \alpha) \deg_D^+(u) + (1/2 - \alpha) \deg_D^-(u)
$$ 
arcs which are incident with $u$, contrary to our assumption. 

By linearly ordering the elements of $I_u^{(iii)}$ (in the natural way), and keeping every fifth element we clearly end up with a set $I_u^{(iv)} \subseteq I_u^{(iii)}$ of size $|I_u^{(iv)}| \geq \alpha r/20$ which satisfies Properties (ii), (iii) and (iv) above. Finally, it follows by Property (a) above that there are less than $\alpha r/40$ indices $1 \leq i \leq r$ for which $u$ is $i$-bad. In particular, it follows that there exists a set $I_u \subseteq I_u^{(iv)}$ of size $|I_u| \geq \alpha r/20 - \alpha r/40 = \alpha r/40$ which satisfies Properties (ii), (iii), (iv) and (v) above.    
\end{proof}

We are now ready to describe the different types of steps we will use to build a long path in the next section. Each such step will consist of an arc $(x,y)$ where both $x$ and $y$ exhibit certain desirable properties. We thus start by describing such vertices.

\begin{definition} \label{def::niceVertex}
Let $\varepsilon, \varepsilon'$ be positive real numbers. Let $D=(V,E)$ be a digraph on $n$ vertices, let $q_1 = q_1(n)$, $q_2 = q_2(n)$, $\ell = \ell(n)$ and $r \leq n/\ell$ be positive integers, and let $0 < d = d(n) \leq 1$. Let $V_1, \ldots, V_r$ be pairwise disjoint subsets of $V$, of size $\ell$ each, such that, for every $1 \leq i \leq r$, the pair $(V_i, V_{i+1})$ is $(\varepsilon)$-regular with directed density $d$. Let $X \subseteq V$ be some set and let $1 \leq s \leq r$. A vertex $x \in V_s \setminus X$ is called
\begin{description}
\item [(i)] nice with respect to $X$ (and $D$, $\ell$, $V_1, \ldots, V_r$, $\varepsilon$, $\varepsilon'$, $d$, $q_1$ and $q_2$) if $q_2 \geq \deg_D^+(x, V_{s+1} \setminus X) \geq (1 - \varepsilon')(1 - \varepsilon) d |V_{s+1} \setminus X| \geq q_1$. 
\item [(ii)] backwards nice with respect to $X$ if $q_2 \geq \deg_D^-(x, V_{s-1} \setminus X) \geq (1 - \varepsilon')(1 - \varepsilon) d |V_{s-1} \setminus X| \geq q_1$.
\item [(iii)] very nice with respect to $X$ if it is both nice and backwards nice. 
\end{description}
\end{definition} 
The purpose of the set $X$ in the above definition (and in the next few definitions and lemmas) is to make this definition more flexible. This will be useful in the next section, where we will want to use certain properties of nice vertices with respect to a set $X$ which will constantly change.

We will make use of the following three types of basic steps.

\begin{definition} \label{def::stepTypes}
Let $D$, $\ell$, $V_1, \ldots, V_r$, $\varepsilon$, $\varepsilon'$, $d$, $q_1$, $q_2$, $X$, $s$ and $x$ be as in Definition~\ref{def::niceVertex}.  
\begin{description}
\item [(i)] A standard forward step from $x$ with respect to $X$ is an arc $(x,y) \in E(D)$ such that $y \in V_{s+1} \setminus X$ is nice. 
\item [(ii)] A random forward step from $x$ with respect to $X$ is an arc $(x,y) \in E(D)$ such that $y$ is chosen uniformly at random among all nice vertices of $N_D^+(x, V_{s+1} \setminus X)$.
\item [(iii)] A standard backward step from $x$ with respect to $X$ is an arc $(y,x) \in E(D)$ such that $y \in V_{s-1} \setminus X$ is backwards nice.
\end{description}
\end{definition}

Next, we describe sufficient conditions for such steps to exist.

\begin{lemma} \label{lem::stepsExist}
Let $D$, $\ell$, $V_1, \ldots, V_r$, $\varepsilon$, $\varepsilon'$, $d$, $q_1$, $q_2$, $X$, $s$ and $x$ be as in Definition~\ref{def::niceVertex}. Assume further that $q_1 \leq (1 - \varepsilon')(1 - \varepsilon) d \varepsilon' \ell$, that $|V_i \cap X| \leq (1 - \varepsilon') \ell$ holds for every $1 \leq i \leq r$ and that $\deg_D^+(v, V_i) \leq q_2$ holds for every $v \in V(D) \setminus X$ and every $1 \leq i \leq r$. If $x \in V_s \setminus X$ is a nice vertex and $x$ is not bad of type I with respect to $V_{s+1}, V_{s+2}$, then there exists a nice vertex $y \in N_D^+(x, V_{s+1} \setminus X)$, that is, there exists a standard forward step from $x$. Similarly, if $x$ is a backwards nice vertex and $x$ is not bad of type I with respect to $V_{s-2}, V_{s-1}$, then there exists a nice vertex $y \in N_D^-(x, V_{s-1} \setminus X)$, that is, there exist a standard backward step from $x$.  
\end{lemma}

\begin{proof}
We will prove the existence of a standard forward step; the existence of a standard backward step can be proved analogously. Let $Y = N_D^+(x, V_{s+1} \setminus X)$. Since $x$ is nice, it follows that $q_2 \geq |Y| \geq (1 - \varepsilon')(1 - \varepsilon) d |V_{s+1} \setminus X| \geq q_1$. Since, moreover, $x$ is not bad of type I with respect to $V_{s+1}, V_{s+2}$ and the pair $(V_{s+1}, V_{s+2})$ is $(\varepsilon)$-regular with directed density $d$, it follows that there exists a set $Z \subseteq Y$ such that $|Z| \geq (1 - \varepsilon')|Y|$ and the pair $(Z, V_{s+2})$ is $(\varepsilon')$-regular with directed density $d'$ for some $d' \geq (1 - \varepsilon) d$. Since $|V_{s+2} \setminus X| \geq \varepsilon' \ell$, it follows by Lemma~\ref{smallDegreeVertices} that there exists a vertex $y \in Z$ such that $\deg_D^+(y, V_{s+2} \setminus X) \geq (1 - \varepsilon')(1 - \varepsilon) d |V_{s+2} \setminus X| \geq q_1$. Since $y \notin X$, it follows from our assumption that $\deg_D^+(y, V_{s+2} \setminus X) \leq q_2$ and thus $y$ is nice.   
\end{proof}

We will also make use of the following composite steps which consist of several simple ones. The first of these consists of six arcs and is used to absorb a specific vertex into a path.

\begin{definition} \label{def::bigStep}
Let $D$, $\ell$, $V_1, \ldots, V_r$, $\varepsilon$, $\varepsilon'$, $d$, $q_1$, $q_2$, $X$, $s$ and $x$ be as in Definition~\ref{def::niceVertex} and let $v \in V(D)$ be a vertex. A big step from $x$ via $v$ with respect to $X$ consists of six arcs $(x, y_1), (y_1, y_2), (y_2, y_3)$, $(y_3, v), (v, y_4)$ and $(y_4, y_5)$ of $E(D)$ such that $y_1, y_2, y_3, y_4, y_5 \in V(D) \setminus X$ and $y_5$ is nice.  
\end{definition}

\begin{lemma} \label{lem::bigstepExists}
Let $D$, $\ell$, $V_1, \ldots, V_r$, $\varepsilon$, $\varepsilon'$, $d$, $q_1$, $q_2$, $X$, $s$, $x$ and $v$ be as in Definition~\ref{def::bigStep}. Let $I_v = N^-_D(v, V \setminus X)$, $\bar{I}_v = N^-_D(I_v, V \setminus X)$, $O_v = N^+_D(v, V \setminus X)$ and $\bar{O}_v = N^+_D(O_v, V \setminus X)$. Assume that $|V_{s+2} \cap \bar{I}_v| \geq \ell/3$ and that there exists $1 \leq j \leq r$ such that $|V_j \cap \bar{O}_v| \geq \ell/3$. Assume further that $q_1 \leq (1 - \varepsilon')(1 - \varepsilon) d \varepsilon' \ell$, that $\varepsilon' \leq 0.01$, that $|V_i \cap X| \leq \ell/4$ holds for every $1 \leq i \leq r$, and that $\deg_D^+(w, V_i) \leq q_2$ holds for every $w \in V(D) \setminus X$ and every $1 \leq i \leq r$. If $x \in V_s \setminus X$ is a nice vertex and $x$ is not bad of type I with respect to $V_{s+1}, V_{s+2}$, then there exists a big step from $x$ via $v$.
\end{lemma}

\begin{proof}
Let $Y = N_D^+(x, V_{s+1} \setminus X)$. Since $x$ is nice, it follows that $q_2 \geq |Y| \geq (1 - \varepsilon')(1 - \varepsilon) d |V_{s+1} \setminus X| \geq q_1$, where the last inequality holds since $q_1 \leq (1 - \varepsilon')(1 - \varepsilon) d \varepsilon' \ell$ and $|V_{s+1} \setminus X| \geq 3\ell/4$. Since, moreover, $x$ is not bad of type I with respect to $V_{s+1}, V_{s+2}$ and the pair $(V_{s+1}, V_{s+2})$ is $(\varepsilon)$-regular with directed density $d$, it follows that there exists a set $Z \subseteq Y$ such that $|Z| \geq (1 - \varepsilon')|Y|$ and the pair $(Z, V_{s+2})$ is $(\varepsilon')$-regular with directed density $d'$ for some $d' \geq (1 - \varepsilon) d$. Since $|X \cap V_{s+2}| \leq \ell/4$ and $|V_{s+2} \cap \bar{I}_v| \geq \ell/3$, it follows that $|(V_{s+2} \cap \bar{I}_v) \setminus X| \geq \varepsilon' \ell$. It thus follows by Lemma~\ref{smallDegreeVertices} that there exists a vertex $y_1 \in Z$ such that $\deg_D^+(y_1, (V_{s+2} \cap \bar{I}_v) \setminus X) \geq (1 - \varepsilon')(1 - \varepsilon) d |(V_{s+2} \cap \bar{I}_v) \setminus X| \geq q_1$. In particular, there exists a vertex $y_2 \in N_D^+(y_1, (V_{s+2} \cap \bar{I}_v) \setminus X)$. By definition of $\bar{I}_v$ and $I_v$ there exists a vertex $y_3 \in N_D^+(y_2, I_v \setminus X)$. Since $y_3 \in I_v$, we have $(y_3, v) \in E(D)$. Since $|X \cap V_j| \leq \ell/4$ and $|V_j \cap \bar{O}_v| \geq \ell/3$, it follows that $|(V_j \cap \bar{O}_v) \setminus X| \geq \varepsilon' \ell$. Since, moreover, $(V_j, V_{j+1})$ is $(\varepsilon)$-regular with directed density $d$, it follows by Lemma~\ref{smallDegreeVertices} that there exists a vertex $y_5 \in (V_j \cap \bar{O}_v) \setminus X$ such that $\deg_D^+(y_5, V_{j+1} \setminus X) \geq (1 - \varepsilon')(1 - \varepsilon) d |V_{j+1} \setminus X| \geq q_1$, where the last inequality holds since $q_1 \leq (1 - \varepsilon')(1 - \varepsilon) d \varepsilon' \ell$ and $|V_{j+1} \setminus X| \geq 3\ell/4$. Since $y_5 \notin X$, it follows that $\deg_D^+(y_5, V_{j+1} \setminus X) \leq q_2$ and thus $y_5$ is nice. Finally, by the definition of $O_v$ and $\bar{O}_v$, there exists $y_4 \in N_D^-(y_5, O_v \setminus X)$.     
\end{proof}

\vspace*{-1.3cm}

\begin{figure}[H] \label{fig1}
\begin{center}
\includegraphics[height=7.7in, width=7.7in]{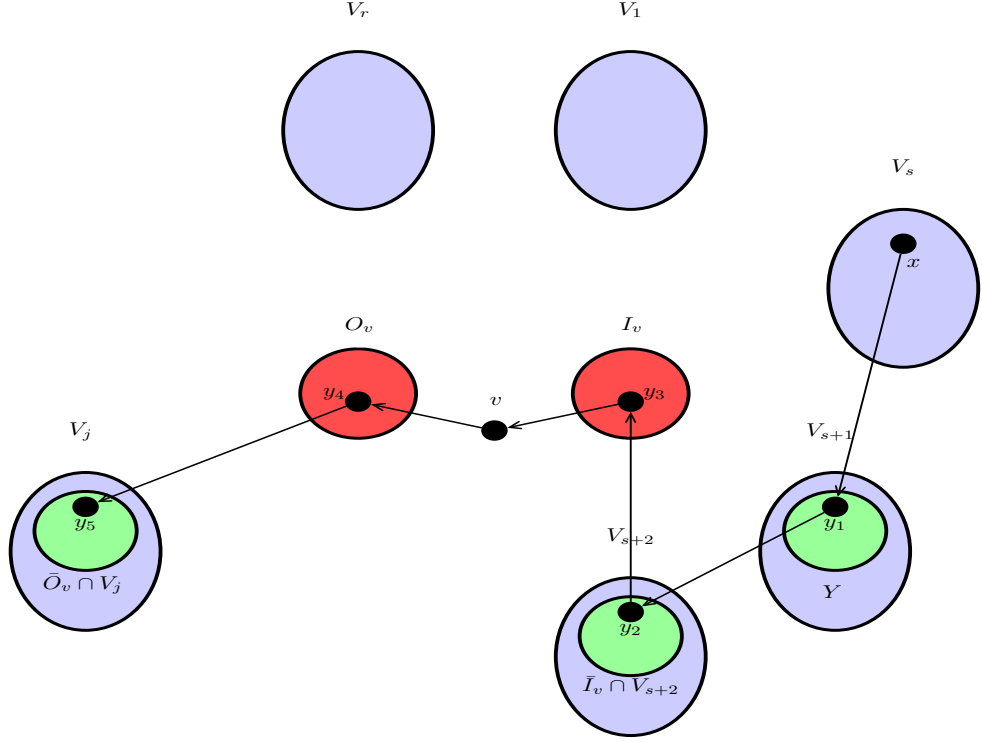}
\vspace*{-5.5cm}
\caption{An example of a big step from $x$ via $v$; note that $v$, $I_v$ and $O_v$ appear in the center of the cycle to stress that we make no assumptions regarding their location.}
\end{center}
\end{figure}

The second composite step we consider consists of four arcs and is used to close a path into a cycle.

\begin{definition} \label{def::closingStep}
Let $D$, $\ell$, $V_1, \ldots, V_r$, $\varepsilon$, $\varepsilon'$, $d$, $q_1$, $q_2$, $X$, $s$ and $x$ be as in Definition~\ref{def::niceVertex} and let $z \in V_{s+4}$ be a vertex. A closing step from $x$ to $z$ with respect to $X$ consists of four arcs $(x, y_1), (y_1, y_2), (y_2, y_3)$ and $(y_3, z)$ of $E(D)$ such that $y_i \in V_{s+i} \setminus X$ for every $1 \leq i \leq 3$.  
\end{definition}

\begin{lemma} \label{lem::closingStep}
Let $D$, $\ell$, $V_1, \ldots, V_r$, $\varepsilon$, $\varepsilon'$, $d$, $q_1$, $q_2$, $X$, $s$, $x$ and $z$ be as in Definition~\ref{def::closingStep}. Assume that $x$ is a nice vertex and is not bad of type I with respect to $V_{s+1}, V_{s+2}$. Assume further that $z$ is not bad of type I with respect to $V_{s+2}, V_{s+3}$ and $\deg_D^-(z, V_{s+3} \setminus X) \geq q_1$. Finally, assume that $q_1 \leq (1 - \varepsilon')(1 - \varepsilon) d \varepsilon' \ell$ and that $|V_i \cap X| < (1 - 2\varepsilon') \ell$ holds for every $1 \leq i \leq r$. Then there exists a closing step from $x$ to $z$ with respect to $X$.     
\end{lemma}

\begin{proof}
Let $Y_1 = N_D^+(x, V_{s+1} \setminus X)$. Since $x$ is nice, it follows that $q_2 \geq |Y_1| \geq (1 - \varepsilon')(1 - \varepsilon) d |V_{s+1} \setminus X| \geq q_1$, where the last inequality holds since $q_1 \leq (1 - \varepsilon')(1 - \varepsilon) d \varepsilon' \ell$ and $|V_{s+1} \setminus X| \geq 2 \varepsilon' \ell$. Since, moreover, $x$ is not bad of type I with respect to $V_{s+1}, V_{s+2}$ and the pair $(V_{s+1}, V_{s+2})$ is $(\varepsilon)$-regular with directed density $d$, it follows that there exists a set $Z_1 \subseteq Y_1$ such that $|Z_1| \geq (1 - \varepsilon')|Y_1|$ and the pair $(Z_1, V_{s+2})$ is $(\varepsilon')$-regular with directed density $d'$ for some $d' \geq (1 - \varepsilon) d$. Similarly, since $\deg_D^-(z, V_{s+3} \setminus X) \geq q_1$, since $z$ is not bad of type I with respect to $V_{s+2}, V_{s+3}$ and since the pair $(V_{s+2}, V_{s+3})$ is $(\varepsilon)$-regular with directed density $d$, it follows that there exists a set $Z_3 \subseteq N_D^-(z, V_{s+3} \setminus X)$ such that $|Z_3| \geq (1 - \varepsilon') q_1$ and the pair $(V_{s+2}, Z_3)$ is $(\varepsilon')$-regular with directed density $d'$ for some $d' \geq (1 - \varepsilon) d$. Since $|V_{s+2} \setminus X| > 2\varepsilon' \ell$ and the pair $(Z_1, V_{s+2})$ is $(\varepsilon')$-regular with positive density, it follows that $|N_D^+(Z_1, V_{s+2} \setminus X)| \geq |V_{s+2} \setminus X| - \varepsilon' \ell > |V_{s+2} \setminus X|/2$. Similarly, $|N_D^-(Z_3, V_{s+2} \setminus X)| > |V_{s+2} \setminus X|/2$ as well and thus $N_D^+(Z_1, V_{s+2} \setminus X) \cap N_D^-(Z_3, V_{s+2} \setminus X) \neq \emptyset$. Choosing any vertices $y_2 \in N_D^+(Z_1, V_{s+2} \setminus X) \cap N_D^-(Z_3, V_{s+2} \setminus X)$, $y_1 \in N_D^-(y_2, Z_1)$ and $y_3 \in N_D^+(y_2, Z_3)$ completes the proof.    
\end{proof}

\bigskip

We end this section by proving some properties of random forward steps. More precisely, we show that regularity and the fact that we are working with a subdigraph of a random digraph, imply that the probabilities that vertices of a sequences of two (or more) consecutive random forward steps, starting at a nice vertex, belong to predefined sets is essentially uniformly distributed (up to constant factors).

\begin{lemma} \label{lem::twoSteps}
Let $D$, $\ell$, $V_1, \ldots, V_r$, $\varepsilon$, $\varepsilon'$, $d$, $q_1$, $q_2$, $X$ and $s$ be as in Definition~\ref{def::niceVertex} and let $x \in V_s \setminus X$ be a nice vertex which is not bad of type I with respect to $V_{s+1}, V_{s+2}$. Assume further that $\varepsilon \leq \varepsilon' \leq 10^{-3}$, that $q_1 \leq (1 - \varepsilon')(1 - \varepsilon) d \varepsilon' \ell$, that $|V_i \cap X| \leq (1 - \varepsilon') \ell$ holds for every $1 \leq i \leq r$ and that $\deg_D^+(v, V_i) \leq q_2$ holds for every $v \in V(D) \setminus X$ and every $1 \leq i \leq r$. Let $Z \subseteq V_{s+2} \setminus X$ be an arbitrary set of size $|Z| \geq 2 \varepsilon' \ell$. Let $(x,y)$ and $(y,z)$ be two consecutive random forward steps. Then $Pr(z \in Z) \geq \frac{0.99 |Z| - \varepsilon' \ell}{|V_{s+2} \setminus X|}$.
\end{lemma}

\begin{proof}
Let $N = \{w \in Z : w \textrm{ is nice}\}$. Since the pair $(V_{s+2}, V_{s+3})$ is $(\varepsilon)$-regular with directed density $d$, $q_1 \leq (1 - \varepsilon')(1 - \varepsilon) d \varepsilon' \ell$ and $|V_{s+3} \setminus X| \geq \varepsilon' \ell \geq \varepsilon \ell$, it follows by Lemma~\ref{smallDegreeVertices} that $|N| \geq |Z| - \varepsilon \ell \geq |Z| - \varepsilon' \ell \geq \varepsilon' \ell$.

Let $Y = N_D^+(x, V_{s+1} \setminus X)$. Since $x$ is nice, it follows that $q_2 \geq |Y| \geq (1 - \varepsilon')(1 - \varepsilon) d |V_{s+1} \setminus X| \geq q_1$, where the last inequality holds since $q_1 \leq (1 - \varepsilon')(1 - \varepsilon) d \varepsilon' \ell$ and $|V_{s+1} \setminus X| \geq \varepsilon' \ell$. Since, moreover, $x$ is not bad of type I with respect to $V_{s+1}, V_{s+2}$ and the pair $(V_{s+1}, V_{s+2})$ is $(\varepsilon)$-regular with directed density $d$, it follows that there exists a set $Y' \subseteq Y$ such that $|Y'| \geq (1 - \varepsilon')|Y|$ and the pair $(Y', V_{s+2})$ is $(\varepsilon')$-regular with directed density $d'$ for some $(1 - \varepsilon) d \leq d' \leq (1 + \varepsilon) d$. Let $Y_1 = \{w \in Y' : w \textrm{ is nice}\}$, let $Y_2 = \{w \in Y' : \deg_{D'}^+(w, V_{s+2} \setminus X) \leq (1 + \varepsilon') (1 + \varepsilon) d |V_{s+2} \setminus X|\}$, let $Y_3 = \{w \in Y' : \deg_{D'}^+(y, N) \geq (1 - \varepsilon') (1 - \varepsilon) d |N|\}$ and let $Y'' = Y_1 \cap Y_2 \cap Y_3$. Since the pair $(Y', V_{s+2})$ is $(\varepsilon')$-regular with directed density $d' \geq (1 - \varepsilon) d$, $q_1 \leq (1 - \varepsilon')(1 - \varepsilon) d \varepsilon' \ell$ and $|V_{s+2} \setminus X| \geq \varepsilon' \ell$, it follows by Lemma~\ref{smallDegreeVertices} that $|Y_1| \geq |Y'| - \varepsilon' |Y'|$. Similarly, since the pair $(Y', V_{s+2})$ is $(\varepsilon')$-regular with directed density $d' \leq (1 + \varepsilon) d$ and $|V_{s+2} \setminus X| \geq \varepsilon' \ell$, it follows by Lemma~\ref{smallDegreeVertices} that $|Y_2| \geq |Y'| - \varepsilon' |Y'|$. Finally, since the pair $(Y', V_{s+2})$ is $(\varepsilon')$-regular with directed density $d' \geq (1 - \varepsilon) d$ and $|N| \geq \varepsilon' \ell$, it follows by Lemma~\ref{smallDegreeVertices} that $|Y_3| \geq |Y'| - \varepsilon' |Y'|$. Therefore, $|Y''| \geq |Y'| - 3 \varepsilon' |Y'| \geq (1 - 3 \varepsilon')(1 - \varepsilon')|Y|$.

We conclude that
\begin{eqnarray*}
Pr(z \in Z) &\geq& Pr(z \in N) \\ 
&\geq& Pr(z \in N \textrm{ and } y \in Y'')\\
&=& Pr(y \in Y'') \cdot Pr(z \in N \mid y \in Y'')\\
&\geq& \left(1 - 3 \varepsilon' \right) \left(1 - \varepsilon' \right) \cdot \frac{(1 - \varepsilon') (1 - \varepsilon) d |N|}{(1 + \varepsilon') (1 + \varepsilon) d |V_{s+2} \setminus X|}\\
&\geq& \left(1 - 3 \varepsilon' \right) \left(1 - \varepsilon' \right) \cdot \frac{(1 - \varepsilon') (1 - \varepsilon) (|Z| - \varepsilon' \ell)}{(1 + \varepsilon') (1 + \varepsilon) |V_{s+2} \setminus X|}\\
&\geq& \frac{0.99 |Z| - \varepsilon' \ell}{|V_{s+2} \setminus X|} \,.
\end{eqnarray*}
\end{proof} 

\begin{lemma} \label{lem::threeSteps}
Let $D$, $\ell$, $V_1, \ldots, V_r$, $\varepsilon$, $\varepsilon'$, $d$, $q_1$, $q_2$, $X$ and $s$ be as in Definition~\ref{def::niceVertex} and let $x \in V_s \setminus X$ be a nice vertex which is not bad of type I with respect to $V_{s+1}, V_{s+2}$. Assume further that $\varepsilon \leq \varepsilon' \leq 10^{-3}$, that $q_1 \leq (1 - \varepsilon')(1 - \varepsilon) d \varepsilon' \ell$, that $|V_i \cap X| \leq 2 \ell/3$ holds for every $1 \leq i \leq r$ and that $\deg_D^+(v, V_i) \leq q_2$ holds for every $v \in V(D) \setminus X$ and every $1 \leq i \leq r$. Let $Z \subseteq V_{s+3} \setminus X$ be a set which satisfies all of the following properties: 
\begin{description}
\item [(a)] $q_2 \geq |Z| \geq q_1$.
\item [(b)] The pair $(V_{s+2}, Z)$ is $(\varepsilon')$-regular with directed density $d(V_{s+2}, Z) \geq (1 - \varepsilon) d$.
\item [(c)] The pair $(Z, V_{s+4})$ is $(\varepsilon')$-regular with directed density $d(Z, V_{s+4}) \geq (1 - \varepsilon) d$. 
\end{description}
Let $(x,y), (y,y')$ and $(y',z)$ be three consecutive random forward steps. Then $Pr(z \in Z) \geq 0.95 |Z|/ \ell$.
\end{lemma}

\begin{proof}
Let $N = \{w \in Z : w \textrm{ is nice}\}$. Since the pair $(Z, V_{s+4})$ is $(\varepsilon')$-regular with directed density $d(Z, V_{s+4}) \geq (1 - \varepsilon) d$, $q_1 \leq (1 - \varepsilon')(1 - \varepsilon) d \varepsilon' \ell$ and $|V_{s+4} \setminus X| \geq \varepsilon' \ell$, it follows by Lemma~\ref{smallDegreeVertices} that $|N| \geq (1 - \varepsilon') |Z| \geq \varepsilon' |Z|$.

Let $Y_1 = \{w \in V_{s+2} \setminus X : \deg_D^+(w, V_{s+3} \setminus X) \leq (1 + \varepsilon) d |V_{s+3} \setminus X|\}$. Since the pair $(V_{s+2}, V_{s+3})$ is $(\varepsilon)$-regular with directed density $d$ and $|V_{s+3} \setminus X| \geq \varepsilon' \ell \geq \varepsilon \ell$, it follows by Lemma~\ref{smallDegreeVertices} that $|Y_1| \geq |V_{s+2} \setminus X| - \varepsilon \ell \geq |V_{s+2} \setminus X| - \varepsilon' \ell$. Let $Y_2 = \{w \in V_{s+2} \setminus X : \deg_D^+(w, N) \geq (1 - \varepsilon') (1 - \varepsilon) d |N|\}$. Since the pair $(V_{s+2}, Z)$ is $(\varepsilon')$-regular with directed density $d(V_{s+2}, Z) \geq (1 - \varepsilon) d$ and $|N| \geq \varepsilon' |Z|$, it follows by Lemma~\ref{smallDegreeVertices} that $|Y_2| \geq |V_{s+2} \setminus X| - \varepsilon' \ell$. Let $Y = Y_1 \cap Y_2$, then $|Y| \geq |V_{s+2} \setminus X| - 2 \varepsilon' \ell$.

It follows by Lemma~\ref{lem::twoSteps} that $Pr(y' \in Y) \geq \frac{0.99 |Y| - \varepsilon' \ell}{|V_{s+2} \setminus X|} \geq 0.97$, where the last inequality holds since $|V_{s+2} \setminus X| \geq \ell/3$ and $\varepsilon' \leq 10^{-3}$.

We conclude that
\begin{eqnarray*}
Pr(z \in Z) &\geq& Pr(z \in N)\\
&\geq& Pr(z \in N \textrm{ and } y' \in Y)\\
&=& Pr(y' \in Y) \cdot Pr(z \in N \mid y' \in Y)\\
&\geq& 0.97 \cdot \frac{(1 - \varepsilon') (1 - \varepsilon) d |N|}{(1 + \varepsilon) d |V_{s+3} \setminus X|}\\
&\geq& 0.97 \cdot \frac{(1 - \varepsilon')^2 (1 - \varepsilon) |Z|}{(1 + \varepsilon) |V_{s+3} \setminus X|}\\
&\geq& 0.95 |Z|/\ell \,.
\end{eqnarray*}
\end{proof}

\begin{lemma} \label{lem::fourSteps}
Let $D$, $\ell$, $V_1, \ldots, V_r$, $\varepsilon$, $\varepsilon'$, $d$, $q_1$, $q_2$, $X$ and $s$ be as in Definition~\ref{def::niceVertex} and let $x \in V_s \setminus X$ be a nice vertex which is not bad of type I with respect to $V_{s+1}, V_{s+2}$. Assume further that $\varepsilon \leq \varepsilon' \leq 10^{-3}$, that $q_1 \leq (1 - \varepsilon')(1 - \varepsilon) d \varepsilon' \ell$, that $|V_i \cap X| \leq 2 \ell/3$ holds for every $1 \leq i \leq r$ and that $\deg_D^+(v, V_i) \leq q_2$ holds for every $v \in V(D) \setminus X$ and every $1 \leq i \leq r$. Let $Z_1 \subseteq V_{s+3} \setminus X$ and $Z_2 \subseteq V_{s+4} \setminus X$ be sets which satisfy all of the following properties:
\begin{description}
\item [(i)] $q_2 \geq |Z_1|, |Z_2| \geq 2 q_1$. 
\item [(ii)] The pair $(V_{s+2}, Z_1)$ is $(\varepsilon')$-regular with directed density $d(V_{s+2}, Z_1) \geq (1 - \varepsilon) d$.
\item [(iii)] The pair $(Z_1, V_{s+4})$ is $(\varepsilon')$-regular with directed density $(1 - \varepsilon) d \leq d(Z_1, V_{s+4}) \leq (1 + \varepsilon) d$.
\item [(iv)] The pair $(Z_2, V_{s+5})$ is $(\varepsilon')$-regular with directed density $d(Z_2, V_{s+5}) \geq (1 - \varepsilon) d$.
\item [(v)] The pair $(Z_1, Z_2)$ is $(\varepsilon')$-regular with directed density $d(Z_1, Z_2) \geq (1 - \varepsilon')^2 d$.
\end{description}
Let $(x,y), (y,z), (z,z')$ and $(z',z'')$ be four consecutive random forward steps. Then 
$$
Pr(z' \in Z_1 \textrm{ and } z'' \in Z_2) \geq \frac{|Z_1| |Z_2|}{2 \ell^2} \,.
$$
\end{lemma}

\begin{proof}
Let $N_2 = \{w \in Z_2 : w \textrm{ is nice}\}$. Since the pair $(Z_2, V_{s+5})$ is $(\varepsilon')$-regular with directed density $d(Z_2, V_{s+5}) \geq (1 - \varepsilon) d$, $q_1 \leq (1 - \varepsilon')(1 - \varepsilon) d \varepsilon' \ell$ and $|V_{s+5} \setminus X| \geq \varepsilon' \ell$, it follows by Lemma~\ref{smallDegreeVertices} that $|N_2| \geq (1 - \varepsilon') |Z_2| \geq \varepsilon' |Z_2|$.

Let $Y_1 = \{w \in Z_1 : \deg_D^+(w, V_{s+4} \setminus X) \leq (1 + \varepsilon') (1 + \varepsilon) d |V_{s+4} \setminus X|\}$. Since the pair $(Z_1, V_{s+4})$ is $(\varepsilon')$-regular with directed density $d(Z_1, V_{s+4}) \leq (1 + \varepsilon) d$ and $|V_{s+4} \setminus X| \geq \varepsilon' \ell$, it follows by Lemma~\ref{smallDegreeVertices} that $|Y_1| \geq |Z_1| - \varepsilon' |Z_1|$. Let $Y_2 = \{w \in Z_1 : \deg_D^+(w, N_2) \geq (1 - \varepsilon')^3 d |N_2|\}$. Since the pair $(Z_1, Z_2)$ is $(\varepsilon')$-regular with directed density $d(Z_1, Z_2) \geq (1 - \varepsilon')^2 d$ and $|N_2| \geq \varepsilon' |Z_2|$, it follows by Lemma~\ref{smallDegreeVertices} that $|Y_2| \geq |Z_1| - \varepsilon' |Z_1|$. Let $N_1 = Y_1 \cap Y_2$, then $|N_1| \geq |Z_1| - 2 \varepsilon' |Z_1| \geq q_1$.

It follows by Properties (i), (ii) and (iii) that the conditions of Lemma~\ref{lem::threeSteps} are satisfied and thus $Pr(z' \in N_1) \geq 0.95 |N_1|/\ell \geq 0.9 |Z_1|/\ell$.

We conclude that
\begin{eqnarray*}
Pr(z' \in Z_1 \textrm{ and } z'' \in Z_2) &\geq& Pr(z' \in N_1 \textrm{ and } z'' \in N_2)\\
&=& Pr(z' \in N_1) \cdot Pr(z'' \in N_2 \mid z' \in N_1)\\
&\geq& \frac{0.9 |Z_1|}{\ell} \cdot \frac{(1 - \varepsilon')^3 d |N_2|}{(1 + \varepsilon') (1 + \varepsilon) d |V_{s+4} \setminus X|}\\
&\geq& \frac{0.9 |Z_1|}{\ell} \cdot \frac{(1 - \varepsilon')^4 |Z_2|}{(1 + \varepsilon') (1 + \varepsilon) |V_{s+4} \setminus X|}\\
&\geq& \frac{|Z_1| |Z_2|}{2 \ell^2} \,.
\end{eqnarray*}
\end{proof}

\begin{lemma} \label{lem::twoStepsUpperBound}
Let $n$ be a positive integer, let $\log n/\sqrt{n} \ll p = p(n) \leq 1$, let $D \in {\mathcal D}(n,p)$ and let $D' = (V,E)$ be a spanning subdigraph of $D$. Let $\ell$, $V_1, \ldots, V_r$, $\varepsilon$, $\varepsilon'$, $d$, $q_1$, $q_2$, $X$ and $s$ be as in Definition~\ref{def::niceVertex} (with respect to $D'$) and let $x \in V_s \setminus X$ be a nice vertex which is not bad of type I with respect to $V_{s+1}, V_{s+2}$. Assume further that $\varepsilon \leq \varepsilon' \leq 10^{-3}$, that $d = \xi p$ for some $0 < \xi \leq 1$, that $q_1 \leq (1 - \varepsilon')(1 - \varepsilon) d \varepsilon' \ell$, that $|V_i \cap X| \leq 2\ell/3$ holds for every $1 \leq i \leq r$ and that $\deg_D^+(v, V_i) \leq q_2$ holds for every $v \in V(D) \setminus X$ and every $1 \leq i \leq r$. Let $Z_1 \subseteq Z_2 \subseteq V_{s+2} \setminus X$ be arbitrary fixed sets of size $|Z_1| = \ell p^{3/2}$ and $|Z_2| = \ell p$. Let $(x,y)$ and $(y,z)$ be two consecutive random forward steps. Then 
\begin{description}
\item [(a)] $Pr(z \in Z_1) \leq 44 \xi^{-2} p$.
\item [(b)] $Pr(z \in Z_2) \leq 44 \xi^{-2} \sqrt{p}$.
\end{description}
\end{lemma}

\begin{proof}
Let $Y = N_{D'}^+(x, V_{s+1} \setminus X)$. Since $x$ is nice, it follows that $q_2 \geq |Y| \geq (1 - \varepsilon')(1 - \varepsilon) d |V_{s+1} \setminus X| \geq q_1$, where the last inequality holds since $q_1 \leq (1 - \varepsilon')(1 - \varepsilon) d \varepsilon' \ell$ and $|V_{s+1} \setminus X| \geq \varepsilon' \ell$. Since, moreover, $x$ is not bad of type I with respect to $V_{s+1}, V_{s+2}$ and the pair $(V_{s+1}, V_{s+2})$ is $(\varepsilon)$-regular with directed density $d$, it follows that there exists a set $Y' \subseteq Y$ such that $|Y'| \geq (1 - \varepsilon')|Y|$ and the pair $(Y', V_{s+2})$ is $(\varepsilon')$-regular with directed density $d'$ for some $(1 - \varepsilon) d \leq d' \leq (1 + \varepsilon) d$. Let $Y'' = \{w \in Y' : w \textrm{ is nice}\}$. Since the pair $(Y', V_{s+2})$ is $(\varepsilon')$-regular with directed density $d' \geq (1 - \varepsilon) d$, $q_1 \leq (1 - \varepsilon')(1 - \varepsilon) d \varepsilon' \ell$ and $|V_{s+2} \setminus X| \geq \varepsilon' \ell$, it follows by Lemma~\ref{smallDegreeVertices} that 
\begin{equation} \label{eq::manyNiceVertices}
|Y''| \geq |Y'| - \varepsilon' |Y'| \geq (1 - \varepsilon')^2 |Y| \geq (1 - \varepsilon')^3 (1 - \varepsilon) d |V_{s+1} \setminus X| \geq (1 - \varepsilon')^3 (1 - \varepsilon) d \ell/3 \geq d \ell/4 \,.
\end{equation}

By the definition of a random forward step, $y$ is a nice vertex and thus 
\begin{equation} \label{eq::yIsNice}
\deg_{D'}^+(y, V_{s+2} \setminus X) \geq (1 - \varepsilon')(1 - \varepsilon) d |V_{s+2} \setminus X| \geq (1 - \varepsilon')(1 - \varepsilon) d \ell/3 \,.
\end{equation}

Let $Z = \{z \in N_{D'}^+(y, V_{s+2} \setminus X) : z \textrm{ is nice}\}$. Since the pair $(V_{s+2}, V_{s+3})$ is $(\varepsilon)$-regular with directed density $d$, $q_1 \leq (1 - \varepsilon')(1 - \varepsilon) d \varepsilon' \ell \leq (1 - \varepsilon')(1 - \varepsilon)^2 d \ell/3$ and $|V_{s+3} \setminus X| \geq \varepsilon \ell$, it follows by Lemma~\ref{smallDegreeVertices} and by~\eqref{eq::yIsNice} that
\begin{equation} \label{eq::largeDegree}
|Z| \geq (1 - \varepsilon')(1 - \varepsilon) d \ell/3 - \varepsilon \ell \geq d \ell/4 \,.
\end{equation}

Let $W_1 = \{w \in Y : \deg_{D'}^+(w, Z_1) \geq 7 \sqrt{p} |Z_1|\}$. It follows by Lemma~\ref{lem::degreeTooHigh} (ii) that $|W_1| \leq \ell p^{3/2}$. 
Hence
\begin{eqnarray*}
Pr(z \in Z_1) &=& Pr(y \notin W_1) Pr(z \in Z_1 \mid y \notin W_1) + Pr(y \in W_1) Pr(z \in Z_1 \mid y \in W_1) \\
&\leq& 1 \cdot \frac{7 \sqrt{p} |Z_1|}{d \ell/4} + \frac{|W_1|}{|Y''|} \cdot \frac{|Z_1|}{d \ell/4} 
\leq \frac{7 \ell p^2}{d \ell/4} + \frac{\ell p^{3/2}}{d \ell/4} \cdot \frac{\ell p^{3/2}}{d \ell/4} 
= \frac{28 p^2}{d} + \frac{16 p^3}{d^2} \\
&\leq& 44 \xi^{-2} p \,,
\end{eqnarray*} 
where the first inequality holds by~\eqref{eq::largeDegree} and the definition of $W_1$ and the second inequality holds by~\eqref{eq::manyNiceVertices}. This proves (a).

For (b), let $W_2 = \{w \in Y : \deg_{D'}^+(w, Z_2) \geq 7 \sqrt{p} |Z_2|\}$. It follows by Lemma~\ref{lem::degreeTooHigh} (ii) that $|W_2| \leq \ell p^{3/2}$. 
Hence
\begin{eqnarray*}
Pr(z \in Z_2) &=& Pr(y \notin W_2) Pr(z \in Z_2 \mid y \notin W_2) + Pr(y \in W_2) Pr(z \in Z_2 \mid y \in W_2) \\
&\leq& 1 \cdot \frac{7 \sqrt{p} |Z_2|}{d \ell/4} + \frac{|W_2|}{|Y''|} \cdot \frac{|Z_2|}{d \ell/4} 
\leq \frac{7 \ell p^{3/2}}{d \ell/4} + \frac{\ell p^{3/2}}{d \ell/4} \cdot \frac{\ell p}{d \ell/4} 
= \frac{28 p^{3/2}}{d} + \frac{16 p^{5/2}}{d^2} \\
&\leq& 44 \xi^{-2} \sqrt{p} \,,
\end{eqnarray*} 
where the first inequality holds by~\eqref{eq::largeDegree} and the definition of $W_2$ and the second inequality holds by~\eqref{eq::manyNiceVertices}.
\end{proof}

\begin{lemma} \label{lem::threeStepsUpperBound}
Let $n$ be a positive integer, let $\log n/\sqrt{n} \ll p = p(n) \leq 1$, let $D \in {\mathcal D}(n,p)$ and let $D' = (V,E)$ be a spanning subdigraph of $D$. Let $\ell$, $V_1, \ldots, V_r$, $\varepsilon$, $\varepsilon'$, $d$, $q_1$, $q_2$, $X$ and $s$ be as in Definition~\ref{def::niceVertex} (with respect to $D'$) and let $x \in V_s \setminus X$ be a nice vertex which is not bad of type I with respect to $V_{s+1}, V_{s+2}$. Assume further that $\varepsilon \leq \varepsilon' \leq 10^{-3}$, that $d = \xi p$ for some $0 < \xi \leq 1$, that $q_1 \leq (1 - \varepsilon')(1 - \varepsilon) d \varepsilon' \ell$, that $|V_i \cap X| \leq 2\ell/3$ holds for every $1 \leq i \leq r$ and that $\deg_D^+(v, V_i) \leq q_2$ holds for every $v \in V(D) \setminus X$ and every $1 \leq i \leq r$. Let $Z \subseteq V_{s+3} \setminus X$ be an arbitrary fixed set of size $|Z| = 2 \ell p$. Let $(x,y), (y,y')$ and $(y',z)$ be three consecutive random forward steps. Then $Pr(z \in Z) \leq 3000 \xi^{-3} p$.
\end{lemma}

\begin{proof}
Let $W_1 = \{w \in V_{s+2} \setminus X : \deg_{D'}^+(w, Z) \geq 7 \sqrt{p} |Z|\}$ and let $W_2 = \{w \in V_{s+2} \setminus X : \deg_{D'}^+(w, Z) \geq 2 p |Z|\}$; clearly $W_1 \subseteq W_2$. It follows by Lemma~\ref{lem::degreeTooHigh} (ii) that $|W_1| \leq \ell p^{3/2}$ and by Lemma~\ref{lem::degreeTooHigh} (i) that $|W_2| \leq \ell p$. Moreover, it follows by Lemma~\ref{lem::twoStepsUpperBound} that $Pr(y' \in W_1) \leq 44 \xi^{-2} p$ and that $Pr(y' \in W_2) \leq 44 \xi^{-2} \sqrt{p}$. Finally, by the definition of a random forward step, $y'$ is a nice vertex and thus 
\begin{equation} \label{eq::yIsNiceAgain}
\deg_{D'}^+(y', V_{s+3} \setminus X) \geq (1 - \varepsilon')(1 - \varepsilon) d |V_{s+3} \setminus X| \geq (1 - \varepsilon')(1 - \varepsilon) d \ell/3 \,.
\end{equation}

Let $W = \{w \in N_{D'}^+(y', V_{s+3} \setminus X) : w \textrm{ is nice}\}$. Since the pair $(V_{s+3}, V_{s+4})$ is $(\varepsilon)$-regular with directed density $d$, $q_1 \leq (1 - \varepsilon')(1 - \varepsilon) d \varepsilon' \ell \leq (1 - \varepsilon')(1 - \varepsilon)^2 d \ell/3$ and $|V_{s+4} \setminus X| \geq \varepsilon \ell$, it follows by Lemma~\ref{smallDegreeVertices} and by~\eqref{eq::yIsNiceAgain} that
\begin{equation} \label{eq::largeDegreeAgain}
|W| \geq (1 - \varepsilon')(1 - \varepsilon) d \ell/3 - \varepsilon \ell \geq d \ell/4 \,.
\end{equation}

We conclude that
\begin{eqnarray*}
Pr(z \in Z) &=& Pr(y' \in W_1) Pr(z \in Z \mid y' \in W_1) + Pr(y' \in W_2 \setminus W_1) Pr(z \in Z \mid y' \in W_2 \setminus W_1) \\ 
&+& Pr(y' \notin W_2) Pr(z \in Z \mid y' \notin W_2) \\
&\leq& 44 \xi^{-2} p \cdot \frac{|Z|}{d \ell/4} + 44 \xi^{-2} \sqrt{p} \cdot \frac{7 \sqrt{p} |Z|}{d \ell/4} + 1 \cdot \frac{2 p |Z|}{d \ell/4} \\
&\leq& 352 \xi^{-3} p + 2464 \xi^{-3} p + 16 \xi^{-1} p \\ 
&\leq& 3000 \xi^{-3} p \,. 
\end{eqnarray*} 
\end{proof}

\section{Proof of the main result}
\label{sec::proofMain}

We start with the upper bound in Theorem~\ref{directedHam}; we will in fact prove the following stronger result. Let $D \in {\mathcal D}(n,p)$, where $p \gg \log n/n$. Then a.a.s. one can delete at most $\left(1/2 + 10 \sqrt{\frac{\log n}{n p}} \right) \deg_D^+(u)$ of the out-going arcs and at most $\left(1/2 + 10 \sqrt{\frac{\log n}{n p}} \right) \deg_D^-(u)$ of the in-going arcs at every vertex $u \in V(D)$ so that the resulting digraph is non-Hamiltonian; note that $10 \sqrt{\frac{\log n}{n p}} = o(1)$ for $p \gg \log n/n$. 

\textbf{Proof of Theorem~\ref{directedHam} (upper bound):}
Let $V = V_1 \cup V_2$ be an arbitrary partition of $[n]$ into two parts of equal size (that is, $||V_1| - |V_2|| \leq 1$). Let $D = ([n],E) \in {\mathcal D}(n,p)$. It follows by Theorem~\ref{th::Chernoff} (iii) and union bound, that a.a.s. every $v \in [n]$ satisfies $|\deg_D^+(v) - n p| \leq 4 \sqrt{n p \log n}$ and $|\deg_D^-(v) - n p| \leq 4 \sqrt{n p \log n}$. Let $u \in V_1$ be an arbitrary vertex, then clearly $\deg^+_D(u,V_2) \sim Bin(|V_2|, p)$. In particular, $(n-1)p/2 \leq \mathbb{E}(\deg^+_D(u,V_2)) \leq (n+1)p/2$. Since a.a.s. $\deg_D^+(u) \geq n p - 4 \sqrt{n p \log n}$, it follows by Theorem~\ref{th::Chernoff} (ii) that 
\begin{eqnarray*}
&& Pr \left(\deg_D^+(u,V_2) \geq \left(1/2 + 10 \sqrt{\frac{\log n}{n p}} \right) \deg_D^+(u) \right)\\ 
&\leq& Pr \left(\deg_D^+(u,V_2) \geq \left(1 + 20 \sqrt{\frac{\log n}{n p}} \right) (np/2 - 2\sqrt{n p \log n}) \right)\\
&\leq& Pr \left(\deg_D^+(u,V_2) \geq \left(1 + 15 \sqrt{\frac{\log n}{n p}} \right) \mathbb{E}(\deg^+(u,V_2)) \right) \\ 
&\leq& e^{- \frac{225 \log n}{3 n p} \cdot \frac{(n-1) p}{2}} = o(1/n) \,. 
\end{eqnarray*}

Taking the union bound over all vertices of $V_1$, we conclude that a.a.s. for every $u \in V_1$ we have $\deg_D^+(u,V_2) \leq \left(1/2 + 10 \sqrt{\frac{\log n}{n p}} \right) \deg_D^+(u)$. An analogous argument shows that a.a.s. $\deg_D^-(w,V_1) \leq \left(1/2 + 10 \sqrt{\frac{\log n}{n p}} \right) \deg_D^-(w)$ for every $w \in V_2$. Our claim now follows since one can obtain a non-Hamiltonian digraph by deleting all arcs of $D$ that are oriented from $V_1$ to $V_2$. In particular, $r_{\ell}({\mathcal D}(n,p), {\mathcal H}) \leq \left (1/2 + 10 \sqrt{\frac{\log n}{n p}} \right) n p$ a.a.s. 
{\hfill $\Box$ \medskip\\}

The remainder of this section is devoted to the proof of the lower bound. Namely, we will prove that a.a.s. even if an adversary deletes
at most $\left(1/2 - \alpha \right) \deg_D^+(u)$ of the out-going arcs and at most $\left(1/2 - \alpha \right) \deg_D^-(u)$ of the in-going arcs at every vertex $u \in V(D)$, where $\alpha > 0$ is an arbitrarily small constant, there is still a Hamilton cycle in the resulting digraph. Let $\varepsilon', \rho, \lambda$ and $\xi$ be positive real numbers and let $m$ be a positive integer such that $m^{-1} \ll \lambda \ll \rho \ll \varepsilon' \ll \xi \ll \alpha$ and moreover $\lambda \ll \xi \varepsilon'$ and $\varepsilon' \ll \alpha^8 \xi^6$, where for positive real numbers $a,b$ the notation $a \ll b$ means that $a/b$ is a sufficiently small real number. Let $\delta = \xi p$, let $\varepsilon_1 = \varepsilon_1(\xi, \lambda, \rho, \varepsilon')$ be the real number whose existence follows from Lemma~\ref{badConfigurationsI}, let $\varepsilon_2 = \varepsilon_2(\xi, \lambda, \rho, \varepsilon')$ be the real number whose existence follows from Lemma~\ref{badConfigurationsII} and let $\varepsilon = \min \{\varepsilon_1, \varepsilon_2\}$. Note that $\varepsilon \leq \varepsilon'$. Let $D \in {\mathcal D}(n,p)$. Note that a.a.s. $|\deg^+_D(u) - n p| \leq 4 \sqrt{n p \log n}$ and $|\deg^-_D(u) - n p| \leq 4 \sqrt{n p \log n}$ hold for every vertex $u \in V(D)$. Hence, we will assume throughout the proof that $D$ satisfies these properties. Fix some $L > 1$ and let $0 < \eta = \eta(m, \varepsilon \xi/2, L) \leq 1$ be the constant whose existence is guaranteed by Lemma~\ref{diregularity}. Since, by Remark~\ref{$D(n,p)$IsBounded}, $D$ is a.a.s. $(\eta, L, p)$-bounded, we assume throughout the proof that it is. Let $D'=(V,E)$ be a digraph obtained from $D$ by deleting at most $\left(1/2 - \alpha \right) \deg_D^+(u)$ of the out-going arcs and at most $\left(1/2 - \alpha \right) \deg_D^-(u)$ of the in-going arcs at every vertex $u \in V(D)$. Note that $D'$ is $(\eta, L, p)$-bounded as well.

Apply Lemma~\ref{diregularity} to $D'=(V,E)$ with parameters $\varepsilon \xi/2$, $L$ and $m$. Let $\{V_0, V_1, \ldots, V_k\}$ be the corresponding $(\varepsilon \xi/2, p)$-regular partition, and let $R = R(D', \delta)$ be the corresponding regularity digraph. It follows by the definition of $R$ that the ordered pair $(V_i, V_j)$ is $(\varepsilon \xi/2, p)$-regular with directed density at least $\delta$ whenever $(v_i, v_j) \in E(R)$. Note that if $d_{D'}(V_i, V_j) \geq p$, then, by Observation~\ref{obs::regularity1}, the pair $(V_i, V_j)$ is $(\varepsilon/2)$-regular. If on the other hand $d_{D'}(V_i, V_j) < p$, then, since $d_{D'}(V_i, V_j) \geq \delta = \xi p$, it follows by Observations~\ref{obs::regularity2} and~\ref{obs::regularity1} that the pair $(V_i, V_j)$ is $(\varepsilon/2)$-regular. It thus follows by Lemma~\ref{lem::GS} that we can assume that $(V_i, V_j)$ is $(\varepsilon)$-regular with directed density $\delta$ for every $1 \leq i \neq j \leq k$ such that $(v_i, v_j) \in E(R)$. Let $\ell$ denote the common size of $V_1, \ldots, V_k$; note that $(1 - \varepsilon)n/k \leq \ell \leq n/k$. 

We first show that $R$ contains an almost spanning cycle; our proof will use the following immediate corollary of a classical theorem of Ghouila-Houri~\cite{GH}.

\begin{theorem} \label{th::CorOfGH}
Let $D$ be a digraph on $n$ vertices. If $\delta^+(D) \geq n/2$ and $\delta^-(D) \geq n/2$, then $D$ admits a directed Hamilton cycle. 
\end{theorem}

\begin{lemma} \label{lem::longCycle}
Conditioned on the properties of ${\mathcal D}(n,p)$ mentioned above, $R$ contains a directed cycle of length $r \geq (1 - 2\sqrt{\varepsilon})k$.
\end{lemma}

\begin{proof}
Let $1 \leq i \leq k$ be an index for which there are at most $\sqrt{\varepsilon}k$ indices $1 \leq j \neq i \leq k$ such that $(V_i,V_j)$ is not $(\varepsilon)$-regular (recall that by Remark~\ref{fewLowDegVertices}, at least $(1 - \sqrt{\varepsilon})k$ indices $1 \leq i \leq k$ have this property). Let $1 \leq j \neq i \leq k$ be such that $(V_i,V_j)$ is an $(\varepsilon)$-regular pair but $(v_i,v_j) \notin E(R)$. Since $n \ll \ell^2 p$, we can use Theorem~\ref{th::Chernoff} (i) and union bound to show that a.a.s. $e_D(V_i, V_j) \geq (1-\alpha/5) \ell^2 p$. Since $(v_i,v_j) \notin E(R)$ even though $(V_i,V_j)$ is $(\varepsilon)$-regular, it must hold that $d_{D'}(V_i,V_j) < \delta$. Hence, recalling that $\delta = \xi p$, we conclude that at least $(1 - \xi - \alpha/5) \ell^2 p$ arcs of $E_D(V_i, V_j)$ were deleted from $D$ in order to obtain $D'$. Recall that $\xi, \sqrt{\varepsilon} \ll \alpha$. If $\deg_R^+(v_i) < k/2 + 2\sqrt{\varepsilon}k$, then at least 
\begin{eqnarray*}
\left(k - 1 - \sqrt{\varepsilon}k - (k/2 + 2\sqrt{\varepsilon}k)\right) \left(1 - \xi - \alpha/5 \right) \ell^2 p 
&\geq& \left(1/2 - 4\sqrt{\varepsilon} - \xi - \alpha/5 \right) k \ell^2 p \\ 
&>& \left(1/2 - \alpha/3 \right) (1 - \varepsilon) n \ell p\\ 
&>& \left(1/2 - \alpha/2 \right) \ell n p
\end{eqnarray*}
arcs of $E_D(V_i, [n] \setminus V_i)$ were deleted from $D$ to obtain $D'$. Since a.a.s. the maximum out-degree of $D$ is at most $n p + 4 \sqrt{n p \log n}$, it follows that there exists some vertex $u \in V_i$ such that strictly more than $\left(1/2 - \alpha \right) \deg_D^+(u)$ out-going arcs which are incident with $u$ in $D$ were deleted to obtain $D'$, contrary to our assumption. Therefore  $\deg_R^+(v_i) \geq k/2 + 2\sqrt{\varepsilon}k$. Since the same argument applies to every $1 \leq i \leq k$ for which there are at most $\sqrt{\varepsilon}k$ indices $1 \leq j \neq i \leq k$ such that $(V_i,V_j)$ is not $(\varepsilon)$-regular, it follows by Remark~\ref{fewLowDegVertices} that at least $(1 - \sqrt{\varepsilon})k$ vertices of $R$ have out-degree at least $k/2 + 2 \sqrt{\varepsilon}k$ each. An analogous argument shows that at least $(1 - \sqrt{\varepsilon})k$ vertices of $R$ have in-degree at least $k/2 + 2 \sqrt{\varepsilon}k$ each.

Let $R'$ be the graph obtained from $R$ by successively deleting vertices whose out-degree or in-degree is strictly smaller than $k/2$. It follows by the previous paragraph that $(1 - 2\sqrt{\varepsilon})k \leq |V(R')| \leq k$. Moreover, $\min \{\delta^+(R'), \delta^-(R')\} \geq k/2 \geq |V(R')|/2$ holds by the definition of $R'$. Applying Theorem~\ref{th::CorOfGH} to $R'$ completes the proof of the lemma.
\end{proof}

Assume without loss of generality that $C_R : v_1, v_2, \ldots, v_r, v_1$ is a cycle of $R$ of length $r \geq (1 - 2 \sqrt{ \varepsilon})k$. Note that it follows from the definition of $R$ that the pair $(V_i, V_{(i \mod r) + 1})$ is $(\varepsilon)$-regular with directed density $\delta$, for every $1 \leq i \leq r$. For the sake of simplicity of presentation we will discard the ``mod $r$'' in the rest of the proof. Hence $V_{i+1}$ will mean $V_1$ in case $i=r$ and $V_{i-1}$ will mean $V_r$ in case $i=1$.

We now show how one can find a Hamilton cycle of $D'$. This is done in four stages. In the first stage we build a path $P_1$ of $D'$ that includes certain ``problematic'' vertices (while certain other ``problematic'' vertices are intentionally avoided and their inclusion is postponed to the fourth stage). In the second stage we extend the path that was built in the first stage to an ``almost spanning'' path $P_2$ such that, for every $u \in V \setminus P_2$, there are ``many'' arcs $(x,y) \in E(P_2)$ for which $(x,u) \in E(D')$ and $(u,y) \in 
E(D')$. In the third stage we close the path that was built in the second stage into a cycle. Finally, in the fourth stage we extend this cycle to a Hamilton cycle by adding all remaining vertices.

\subsection{Stage 1: Absorbing problematic vertices into a short path} 
\label{subsec::stage1}

In this subsection we build a directed path $P_1$ of $D' = (V,E)$ which includes certain problematic vertices (by abuse of notation, $P_1$ will denote the path we build at any point during Stage 1; moreover, we will use $P_1$ to denote the path as well as its vertex set). We begin by describing the different types of problematic vertices we will deal with. 

Let $B$ denote the set of vertices $u \in V$ for which there exists an index $1 \leq i \leq r$ such that $|\deg^+_D(u,V_i) - \ell p| \geq \varepsilon \ell p$ or $|\deg^-_D(u,V_i) - \ell p| \geq \varepsilon \ell p$. It follows by Lemma~\ref{badSet} that a.a.s. $|B| \leq r p^{-1} \log n \leq \rho \ell$.

Let $T_2 \subseteq V$ denote the set of bad vertices of type II (with respect to $D'$, $V_1, \ldots, V_r$, $\ell, \alpha, \varepsilon, \varepsilon', \delta$, $\lambda \ell p$, and $2 \ell p$). It follows by Lemma~\ref{badConfigurationsII} that a.a.s. $|T_2| \leq \rho \ell$. In particular, $|V_i \cap T_2| \leq \rho \ell$ holds for every $1 \leq i \leq r$. 

The vertices of $B \cup T_2$ are the so-called \emph{problematic} vertices we wish to include in $P_1$ in Stage 1. 

When building $P_1$ we will use some of the steps which were defined in the previous section (see Definitions~\ref{def::stepTypes} and~\ref{def::bigStep}). We thus need to avoid bad vertices of type I with respect to appropriate pairs of sets. For every $1 \leq i \leq r$, let $U_i$ denote the set of vertices of $V_i$ which are bad of type I with respect to $D'$, $\ell$, $V_{i+1}$, $V_{i+2}$, $\varepsilon$, $\varepsilon'$, $\delta$, $\lambda \ell p$, and $2 \ell p$. Similarly, for every $1 \leq i \leq r$, let $W_i$ denote the set of vertices of $V_i$ which are bad of type I with respect to $D'$, $\ell$, $V_{i-2}$, $V_{i-1}$, $\varepsilon$, $\varepsilon'$, $\delta$, $\lambda \ell p$, and $2 \ell p$. By Lemma~\ref{badConfigurationsI} we can assume that $|U_i| \leq \rho \ell$ and $|W_i| \leq \rho \ell$ hold for every $1 \leq i \leq r$. Let $T_1 = \left(\bigcup_{i=1}^r U_i \right) \cup \left(\bigcup_{i=1}^r W_i \right)$. It follows that $|V_i \cap T_1| \leq 2 \rho \ell$ holds for every $1 \leq i \leq r$.  

While building $P_1$, we might include in this path many of the neighbors of some vertex $u \in V \setminus P_1$, thus making it hard to include $u$ in the Hamilton cycle we aim to build. In order to avert this problem, as soon as $P_1$ includes too many neighbors of some vertex $u \in V \setminus P_1$, we will declare $u$ to be dangerous and will try to add it to $P_1$. This notion of \emph{dangerous vertices} is made precise by the following definition.  

\begin{definition} \label{def::dangerousVertex}
A vertex $w \in V \setminus P_1$ is called dangerous if $\deg_{D'}(w, (V \setminus \bigcup_{i=1}^r V_i) \cap (B \cup T_1 \cup T_2 \cup P_1)) \geq n p/20$ or there exists some $1 \leq i \leq r$ such that $\deg_{D'}(w, V_i \cap (B \cup T_1 \cup T_2 \cup P_1)) \geq 100 \rho \ell p$.
\end{definition}

At any point during Stage 1, let $L_1 = B \cup T_1 \cup T_2 \cup P_1$ and let $Dg = Dg(P_1)$ denote the set of dangerous vertices. Note that 
\begin{equation} \label{eq::smallL1}
|V_i \cap L_1| \leq 4 \rho \ell + |V_i \cap P_1|
\end{equation}
holds for every $1 \leq i \leq r$ throughout Stage 1.

Moreover, since $T_1 \subseteq \bigcup_{i=1}^r V_i$, it follows that
\begin{equation} \label{eq::totalL1}
|L_1| \leq |P_1| + |B| + |T_2| + \sum_{i=1}^r |V_i \cap T_1| \leq |P_1| + 4 \rho n \,. 
\end{equation}

We are now ready to describe our algorithm for building $P_1$. For every vertex $u \in \bigcup_{j=1}^r V_j$ we denote by $i(u)$ the unique index $1 \leq i \leq r$ such that $u \in V_i$.  

\begin{algorithm}[H]
\caption{Build $P_1$}
\begin{algorithmic} 
\STATE $v_0 \leftarrow$ arbitrary very nice vertex of $V_1$ with respect to $L_1$ 
\STATE $x \leftarrow v_0$
\STATE $A_0 \leftarrow$ arbitrary subset of $N^-_{D'}(v_0, V_r \setminus L_1)$ of size $\lambda \ell p$
\WHILE{$(B \cup T_2) \setminus P_1 \neq \emptyset$ or $Dg \neq \emptyset$}
\IF{$Dg \neq \emptyset$}
\STATE $v \leftarrow$ arbitrary element of $Dg$
\ELSE[$(B \cup T_2) \setminus P_1 \neq \emptyset$]
\STATE $v \leftarrow$ arbitrary element of $(B \cup T_2) \setminus P_1$
\ENDIF
\STATE $(x, v_0, A_0) \leftarrow ADD(v; x, v_0, A_0)$
\ENDWHILE
\end{algorithmic}
\end{algorithm}

\bigskip

\begin{algorithm}
\caption{$ADD(v; x, v_0, A_0)$}
\begin{algorithmic}
\IF{$v \in A_0$}
\STATE $v_0 \leftarrow$ standard backward step from $v_0$ with respect to $L_1$
\STATE $A_0 \leftarrow$ arbitrary subset of $N^-_{D'}(v_0, V_{i(v_0)-1} \setminus L_1)$ of size $\lambda \ell p$
\ENDIF
\STATE $I \leftarrow \{z \in V \setminus (L_1 \cup A_0) : \exists y \in V \setminus (L_1 \cup A_0) \textrm{ such that } (z,y), (y,v) \in E\}$
\STATE $j_1 \leftarrow$ arbitrary element of $[r]$ such that $|I \cap V_{j_1}| \geq \ell/3$ 
\STATE $O \leftarrow \{z \in V \setminus (L_1 \cup A_0) : \exists y \in V \setminus (L_1 \cup A_0) \textrm{ such that } (v,y), (y,z) \in E\}$
\STATE $j_2 \leftarrow$ arbitrary element of $[r]$ such that $|O \cap V_{j_2}| \geq \ell/3$
\WHILE{$i(x) + 2 \neq j_1$}
\STATE $x \leftarrow$ standard forward step from $x$ with respect to $L_1 \cup A_0$
\ENDWHILE
\STATE $x \leftarrow$ big step from $x$ via $v$ with respect to $L_1 \cup A_0$
\RETURN $(x, v_0, A_0)$
\end{algorithmic}
\end{algorithm}

It remains to prove that Algorithm 1 works. Except for the subroutine ADD, the only non-trivial part of Algorithm 1 is the existence of a very nice vertex $v_0$ (from which the existence of the set $A_0$ will readily follow). Let $N_1 = \{u \in V_1 \setminus L_1 : u \textrm{ is nice with respect to } L_1\}$ and let $N_2 = \{u \in V_1 \setminus L_1 : u \textrm{ is backwards nice with respect to } L_1\}$; it suffices to prove that $N_1 \cap N_2 \neq \emptyset$.  Since the pair $(V_1, V_2)$ is $(\varepsilon)$-regular with directed density $\delta$, it follows by~\eqref{eq::smallL1} and by Lemma~\ref{smallDegreeVertices} that $\deg_{D'}^+(x, V_2 \setminus L_1) \geq (1 - \varepsilon) \delta |V_2 \setminus L_1| \geq \lambda \ell p$ holds for all but at most $\varepsilon \ell$ vertices of $V_1 \setminus L_1$. Since $B \subseteq L_1$, it follows that $\deg_{D'}^+(x, V_2 \setminus L_1) \leq 2 \ell p$ holds for every vertex $x \in V_1 \setminus L_1$. We conclude that $|N_1| > \ell/2$. A similar argument shows that $|N_2| > \ell/2$ holds as well and thus $N_1 \cap N_2 \neq \emptyset$ as claimed. Note that $T_1 \subseteq L_1$ and thus $v_0$ is not bad of type I with respect to $V_{r-1}, V_r$ and is not bad of type I with respect to $V_2, V_3$.    

Our next goal is to prove that the subroutine ADD works as well. We will do so under additional assumptions and will then prove that these assumptions hold throughout Stage 1. We first prove that the indices $j_1$ and $j_2$ mentioned in lines 6 and 8 of the subroutine ADD exist. 

\begin{lemma} \label{lem::OutNeighbours}
Let $I_v = N^-_{D'}(v, V \setminus (L_1 \cup A_0))$, $\bar{I}_v = N^-_{D'}(I_v, V \setminus (L_1 \cup A_0))$, $O_v = N^+_{D'}(v, V \setminus (L_1 \cup A_0))$ and $\bar{O}_v = N^+_{D'}(O_v, V \setminus (L_1 \cup A_0))$. If $|P_1| \leq n/100$ and $\deg_{D'}(v, L_1) \leq n p/5$, then there exist indices $1 \leq j_1 \neq j_2 \leq r$, such that $|V_{j_1} \cap \bar{I}_v| \geq \ell/3$ and $|V_{j_2} \cap \bar{O}_v| \geq \ell/3$.  
\end{lemma}

\begin{proof}
We will prove the existence of $j_2$; the existence of $j_1$ can be proved by a similar argument. Since $(1 - o(1))n p \leq \deg^+_D(v) \leq (1 + o(1))n p$, it follows by the definition of $D'$ that $\deg^+_{D'}(v) \geq (1/2 - o(1))n p$. Moreover, since $\deg_{D'}^+(v, L_1) \leq \deg_{D'}(v, L_1) \leq n p/5$, it follows that $|O_v| \geq (1/2 - 1/5 - \lambda - o(1))n p \geq n p/5$. We claim that $|\bar{O}_v| \geq 2n/5$. Indeed, suppose for a contradiction that $Z_v := V \setminus (L_1 \cup A_0 \cup \bar{O}_v)$ is of size at least 
$$
3n/5 - |P_1| - |B \cup T_1 \cup T_2| - |A_0| \geq (3/5 - 1/100 - 4 \rho - \lambda)n \geq 0.58 n \,,
$$ 
where we used~\eqref{eq::totalL1} in the first inequality. It follows by Theorem~\ref{th::Chernoff} (i) and union bound that 
\begin{eqnarray*}
Pr(e_D(O_v, Z_v) < |O_v| \cdot 0.51 n p) &\leq& Pr(e_D(O_v, Z_v) < 0.9 \mathbb{E}(e_D(O_v, Z_v)))\\ 
&<& \binom{n}{|O_v|} \binom{n}{|Z_v|} e^{- c |O_v| |Z_v| p} < 4^n e^{- c' n^2 p^2} = o(1) \,,
\end{eqnarray*}
where $c$ and $c'$ are appropriate constants and the last equality holds by the assumed lower bound on $p$.

It follows that a.a.s. there exists a vertex $u \in O_v$ such that $\deg_D^+(u, Z_v) \geq 0.51 n p$. Since $E_{D'}(O_v, Z_v) = \emptyset$, it follows that, in order to obtain $D'$ from $D$, we have deleted more than $\deg_D^+(u)/2$ edges incident with $u$. This is clearly a contradiction.

Since $r \geq (1 - 2\sqrt{\varepsilon}) k$, it follows that $|V \setminus \bigcup_{i=1}^r V_i| \leq \varepsilon n + 2\sqrt{\varepsilon} n$. We thus conclude that there exists an index $1 \leq j_2 \leq r$, such that $|V_{j_2} \cap \bar{O}_v| \geq \ell/3$ as claimed.
\end{proof}

Next we remark that, since $B \cup T_1 \subseteq L_1 \cup A_0$, since $\lambda \ell p \leq (1- \varepsilon')(1 - \varepsilon) \delta \varepsilon' \ell$, since $v_0$ is nice, and since we end every standard forward step at a nice vertex, the existence of a standard forward step (line 10) follows from Lemma~\ref{lem::stepsExist} provided that $|V_i \cap (L_1 \cup A_0)| \leq (1- \varepsilon') \ell$ holds for every $1 \leq i \leq r$. Similarly, since $v_0$ is backwards nice, and since we end every standard backward step at a backwards nice vertex, the existence of a standard backward step (line 2) follows from Lemma~\ref{lem::stepsExist} under the same conditions. Finally, the existence of a big step (line 12) follows from Lemma~\ref{lem::bigstepExists} provided that $|V_i \cap (L_1 \cup A_0)| \leq \ell/4$ holds for every $1 \leq i \leq r$ and that the conditions of Lemma~\ref{lem::OutNeighbours} are satisfied.  

Therefore, in order to complete the proof of correctness of the subroutine ADD, it suffices to prove that all the conditions mentioned in the previous paragraph hold throughout Stage 1.

\begin{lemma} \label{lem::additionalConditions}
As long as $(B \cup T_2) \setminus P_1 \neq \emptyset$ or $Dg \neq \emptyset$ the following three conditions hold.
\begin{description}
\item [(a)] $|V_s \cap P_1| \leq 20 \rho \ell$ for every $1 \leq s \leq r$;
\item [(b)] $\deg_{D'}(u, V_s \cap L_1) \leq 110 \rho \ell p$ for every $u \in V \setminus P_1$ and every $1 \leq s \leq r$;
\item [(c)] $\deg_{D'}(u, (V \setminus \bigcup_{i=1}^r V_i) \cap L_1) \leq n p/10$ for every $u \in V \setminus P_1$.
\end{description}
\end{lemma} 

Before we prove Lemma~\ref{lem::additionalConditions}, we remark that it suffices to complete the proof of correctness of the subroutine ADD. Indeed, it follows by Condition (a) and~\eqref{eq::smallL1} that $|V_i \cap (L_1 \cup A_0)| \leq \ell/4 \leq (1 - \varepsilon') \ell$ holds for every $1 \leq i \leq r$. Moreover, since $|V \setminus \bigcup_{i=1}^r V_i| \leq \varepsilon n + 2\sqrt{\varepsilon} k \ell \leq \varepsilon n + 2\sqrt{\varepsilon} n$ holds by Lemma~\ref{lem::longCycle}, it follows by Condition (a) that $|P_1| \leq n/100$. Finally, it follows by Conditions (b) and (c) that for every $u \in V \setminus P_1$
$$
\deg_{D'}(u, L_1) = \deg_{D'}(u, (V \setminus \bigcup_{i=1}^r V_i) \cap L_1) + \sum_{i=1}^r \deg_{D'}(u, V_s \cap L_1) \leq np/10 + 110 r \rho \ell p \leq n p/5 \,,
$$  
so both conditions of Lemma~\ref{lem::OutNeighbours} are met.

\textbf{Proof of Lemma~\ref{lem::additionalConditions}}
Suppose for a contradiction that at least one of (a), (b) and (c) is violated at some point during Stage 1, that is, while $(B \cup T_2) \setminus P_1 \neq \emptyset$ or $Dg \neq \emptyset$ still holds. Consider the first moment in which this occurs. We will distinguish between three cases according to which condition is violated first. Before doing so, we will prove that whenever we run the subroutine ADD, we add to $P_1$ only a few vertices from each $V_i$.

\begin{claim} \label{cl::atMost8}
For every $1 \leq i \leq r$, each single call to the subroutine ADD enlarges $|P_1 \cap V_i|$ by at most 8 and $|P_1 \cap (V \setminus \bigcup_{i=1}^r V_i)|$ by at most 3.
\end{claim}   

\begin{proof}
Consider running $ADD(v; x, v_0, A_0)$ once. First, note that by the proof of Lemma~\ref{lem::bigstepExists}, the only vertices of $V \setminus \bigcup_{i=1}^r V_i$ we might add to $P_1$ are $v$ itself, $y_3 \in I_v$ and $y_4 \in O_v$. Next, fix some $1 \leq i \leq r$. We might add 1 vertex of $V_i$ to $P_1$ if we start with a standard backward step from $V_{i+1}$ (this is the new $v_0$). Then, starting at $V_{i(x)}$, we make a series of standard forward steps until we reach $V_{j_1-2}$. This adds to $P_1$ at most one additional vertex of $V_i$. Once we reach $V_{j_1-2}$ we make a big step consisting of the arcs $(x, y_1), (y_1, y_2), (y_2, y_3), (y_3, v), (v, y_4)$ and $(y_4, y_5)$. The claim now follows since clearly $|V_i \cap \{y_1, y_2, y_3, v, y_4, y_5\}| \leq 6$.    
\end{proof}  

We can now return to the main part of the proof. We will make use of the following notation and terminology. A vertex $w \in P_1$ is called \emph{post dangerous} if it was dangerous before it was added to $P_1$. Note that since we do not remove any vertices from $L_1$ when building $P_1$, after adding $w$ to $P_1$, it is still true that $\deg_{D'}(w, (V \setminus \bigcup_{i=1}^r V_i) \cap L_1) \geq n p/20$ or there exists some $1 \leq i \leq r$ such that $\deg_{D'}(w, V_i \cap L_1) \geq 100 \rho \ell p$. A vertex $w \in P_1$ is called \emph{special} if it was added to $P_1$ when the subroutine ADD was called with $v \in Dg$. At any point during Stage 1, we denote by $Pd = Pd(P_1)$ the set of post dangerous vertices and by $Sp = Sp(P_1)$ the set of special vertices. Note that at any point during Stage 1, $Pd \subseteq Sp \subseteq P_1$ and $Dg \cap Pd = \emptyset$ as, by Definition~\ref{def::dangerousVertex}, once a vertex is added to $P_1$, it is no longer dangerous. Moreover 
\begin{equation} \label{eq::PdSp1}
|Pd| \geq |(V \setminus \bigcup_{i=1}^r V_i) \cap Sp|/3
\end{equation} 
and 
\begin{equation} \label{eq::PdSp2}
|Pd| \geq |V_i \cap Sp|/8 \textrm{ for every } 1 \leq i \leq r \,.    
\end{equation}  
hold by Claim~\ref{cl::atMost8}.

Similarly, it follows from Claim~\ref{cl::atMost8} and the aforementioned bounds on $|B|$ and $|T_2|$ that
\begin{equation} \label{eq::P1minusSp}
|(P_1 \cap V_i) \setminus Sp| \leq 8 |B \cup T_2| \leq 10 \rho \ell \textrm{ for every } 1 \leq i \leq r \,.    
\end{equation}

Starting with Condition (a), consider the moment it is violated for the first time. It follows by Claim~\ref{cl::atMost8} that, at this point, $|V_s \cap P_1| > 20 \rho \ell$ holds for some $1 \leq s \leq r$ and $|V_i \cap P_1| \leq 20 \rho \ell + 8$ holds for every $1 \leq i \leq r$. It follows by~\eqref{eq::P1minusSp} that $|V_s \cap Sp| \geq 10 \rho \ell$ and thus $|Pd| \geq |V_s \cap Sp|/8 \geq \rho \ell$ by~\eqref{eq::PdSp2}. For every $w \in Pd$, it follows by the definitions of $Dg$ and $Pd$ that $\deg_{D'}(w, (V \setminus \bigcup_{i=1}^r V_i) \cap L_1) \geq n p/20$ or there is some $1 \leq j \leq r$ for which $\deg_{D'}(w, V_j \cap L_1) \geq 100 \rho \ell p$. Assume first that there exists a set $F_1 \subseteq Pd$ of size $|F_1| = (r+1)^{-1} \rho \ell$ such that $\deg_{D'}(w, (V \setminus \bigcup_{i=1}^r V_i) \cap L_1) \geq n p/20$ holds for every $w \in F_1$. It follows by Theorem~\ref{th::Chernoff} (ii) that a.a.s. 
\begin{equation} \label{eq::arcsF1}
e_D(F_1) \leq 4 \binom{|F_1|}{2} p \leq 2 (r+1)^{-2} \rho^2 \ell^2 p
\end{equation} 
and the number of arcs of $D'$ with one endpoint in $F_1$ and the other in $((V \setminus \bigcup_{i=1}^r V_i) \cap L_1) \setminus F_1$ is at most
\begin{equation} \label{eq::arcsF1small}
4 |F_1| |(V \setminus \bigcup_{i=1}^r V_i) \cap L_1| p \leq 4 \cdot (r+1)^{-1} \rho \ell \cdot (\varepsilon n + 2\sqrt{\varepsilon} n) \cdot p \leq (r+1)^{-1} \rho \ell n p/100 \,. 
\end{equation}
On the other hand, it follows by~\eqref{eq::arcsF1} and the definition of $F_1$ that the number of arcs of $D'$ with one endpoint in $F_1$ and the other in $((V \setminus \bigcup_{i=1}^r V_i) \cap L_1) \setminus F_1$ is at least 
\begin{equation} \label{eq::arcsF1large}
|F_1| \cdot n p/20 - 2e_{D'}(F_1) \geq (r+1)^{-1} \rho \ell n p/20 - 4 (r+1)^{-2} \rho^2 \ell^2 p > (r+1)^{-1} \rho \ell n p/100
\end{equation}
contrary to~\eqref{eq::arcsF1small}. 

Assume then that no such set $F_1$ exists. It follows that there exist some $1 \leq j \leq r$ and a set $F_2 \subseteq Pd$ of size $|F_2| = (r+1)^{-1} \rho \ell$ such that $\deg_{D'}(w, V_j \cap L_1) \geq 100 \rho \ell p$ holds for every $w \in F_2$. It follows by 
Theorem~\ref{th::Chernoff} (ii) that a.a.s. 
\begin{equation} \label{eq::arcsF2}
e_D(F_2) \leq 4 \binom{|F_2|}{2} p \leq 2 (r+1)^{-2} \rho^2 \ell^2 p
\end{equation}
and the number of arcs of $D'$ with one endpoint in $F_2$ and the other in $(V_j \cap L_1) \setminus F_2$ is at most
\begin{equation} \label{eq::arcsF2small}
3 |F_2| |V_j \cap L_1| p \leq 3 \cdot (r+1)^{-1} \rho \ell \cdot 25 \rho \ell \cdot p = 75 (r+1)^{-1} \rho^2 \ell^2 p \,, 
\end{equation}
where the inequality in~\eqref{eq::arcsF2small} holds by~\eqref{eq::smallL1} and since we consider the moment in which Condition (a) is violated for the first time.
 
On the other hand, it follows by~\eqref{eq::arcsF2} and the definition of $F_2$ that the number of arcs of $D'$ with one endpoint in $F_2$ and the other in $(V_j \cap L_1) \setminus F_2$ is at least
\begin{equation} \label{eq::arcsF2large}
|F_2| \cdot 100 \rho \ell p - 2e_{D'}(F) \geq 100 (r+1)^{-1} \rho^2 \ell^2 p - 4 (r+1)^{-2} \rho^2 \ell^2 p > 75 (r+1)^{-1} \rho^2 \ell^2 p
\end{equation}
contrary to~\eqref{eq::arcsF2small}.

Next, assume that Condition (b) is the first to be violated; let $u \in V \setminus P_1$ and $1 \leq i \leq r$ be such that $\deg_{D'}(u, V_i \cap L_1) > 110 \rho \ell p$. Consider the moment at which $\deg_{D'}(u, V_i \cap L_1) \geq 100 \rho \ell p$ was first satisfied, that is, the moment in which $u$ first became dangerous. Since we did not add $u$ to $P_1$ even though it became dangerous, it follows that from this moment until the time $\deg_{D'}(u, V_i \cap L_1) > 110 \rho \ell p$ first occurred we added to $P_1$ only special vertices (see lines 5-6 of Algorithm 1). In particular, we added at least $10 \rho \ell p$ special vertices to $V_i \cap L_1$. Since, 
$|Pd| \geq |V_i \cap Sp|/8$ by~\eqref{eq::PdSp2}, it follows that $|Pd| \geq \rho \ell p$. Assume first that there exists a set $F'_1 \subseteq Pd$ of size $|F'_1| = (r+1)^{-1} \rho \ell p$ such that $\deg_{D'}(w, (V \setminus \bigcup_{i=1}^r V_i) \cap L_1) \geq n p/20$ holds for every $w \in F'_1$. Similar calculations to the ones in~\eqref{eq::arcsF1},~\eqref{eq::arcsF1small} and~\eqref{eq::arcsF1large} show that a.a.s. the number of arcs of $D'$ with one endpoint in $F'_1$ and the other in $((V \setminus \bigcup_{i=1}^r V_i) \cap L_1) \setminus F'_1$ is on the one hand at most $(r+1)^{-1} \rho \ell n p^2/100$ and on the other hand strictly larger than $(r+1)^{-1} \rho \ell n p^2/100$. This is clearly a contradiction.

Assume then that no such set $F'_1$ exists. It follows that there exist some $1 \leq j \leq r$ and a set $F'_2 \subseteq Pd$ of size $|F'_2| = (r+1)^{-1} \rho \ell p$ such that $\deg_{D'}(w, V_j \cap L_1) \geq 100 \rho \ell p$ holds for every $w \in F'_2$. Similar calculations to the ones in~\eqref{eq::arcsF2},~\eqref{eq::arcsF2small} and~\eqref{eq::arcsF2large} show that a.a.s. the number of arcs of $D'$ with one endpoint in $F'_1$ and the other in $(V_j \cap L_1) \setminus F'_1$ is on the one hand at most $75 (r+1)^{-1} \rho^2 \ell^2 p^2$ and on the other hand strictly larger than $75 (r+1)^{-1} \rho^2 \ell^2 p^2$. This is clearly a contradiction.

The proof that Condition (c) is not violated as long as $(B \cup T_2) \setminus P_1 \neq \emptyset$ or $Dg \neq \emptyset$ is essentially the same as the proof for (b); we omit the straightforward details.
{\hfill $\Box$ \medskip\\} 

While building $P_1$ we may have performed several backward steps. Since it would be convenient later on to assume that $P_1$ starts at $V_1$, at the end of Stage 1 we cyclically shift the labels of $V_1, \ldots, V_r$ such that $v_0 \in V_1$ holds again. We conclude this subsection with a summary of what we have proved; this will be convenient in the next subsection.
 
\begin{proposition} \label{prop::summary}
By the end of Stage 1 we have built a directed path $P_1$ from $v_0 \in V_1$ to $x \in V_s \setminus (A_0 \cup B \cup T_1 \cup T_2)$ for some $1 \leq s \leq r$ which satisfies all of the following properties:
\begin{enumerate}
\item $B \cup T_2 \subseteq P_1$.
\item There exists a set $A_0 \subseteq N_{D'}^-(v_0, V_r \setminus (B \cup T_1 \cup T_2 \cup P_1))$ of size $\lambda \ell p$.
\item $|V_i \cap P_1| \leq 20 \rho \ell$ for every $1 \leq i \leq r$.
\item $\deg_{D'}(u, V_i \cap (L_1 \cup A_0)) \leq 110 \rho \ell p + \lambda \ell p$ for every $u \in V \setminus P_1$ and every $1 \leq i \leq r$.
\item $\deg_{D'}(u, (V \setminus \bigcup_{i=1}^r V_i) \cap L_1) \leq n p/10$ for every $u \in V \setminus P_1$.
\item $x$ is nice with respect to $L_1 \cup A_0$.
\end{enumerate}
\end{proposition}

\subsection{Stage 2: Extending the path to an almost spanning one} \label{subsec::stage2}

In this subsection we extend $P_1$ to an almost spanning path of $D'$ which satisfies certain desirable properties. Throughout this stage we denote the current path by $P_2$ and let $L_2 = A_0 \cup T_1 \cup P_2$. Initially $P_2 = P_1$. 

\begin{algorithm}
\caption{Extend $P_1$ to an almost spanning path $P_2$}
\begin{algorithmic} 
\STATE $x \leftarrow$ last vertex added to $P_1$ in Stage 1
\WHILE{$|V_i \setminus L_2| > 3 \varepsilon' \ell$ for every $1 \leq i \leq r$}
\STATE $x \leftarrow$ random forward step from $x$ with respect to $L_2$
\ENDWHILE
\end{algorithmic}
\end{algorithm}

\bigskip

The correctness of Algorithm 3 follows immediately from Lemma~\ref{lem::stepsExist} and since the last vertex added to $P_1$ in Stage 1 was nice and was not in $T_1$.

The remainder of this subsection is dedicated to the proof of the following lemma which will play a crucial role in Stage 4.

\begin{lemma} \label{goodArcs}
Asymptotically almost surely, at the end of Stage 2, $|\{(x,y) \in E(P_2) : (x,u) \in E(D') \text{ and } (u,y) \in E(D')\}| \geq 10^{-10} \alpha^4 \xi^3 p^2 n$ holds for every $u \in V \setminus P_2$.
\end{lemma}

\begin{proof}
Fix some $u \in V \setminus P_2$ and let $I_u \subseteq [r]$ be a set satisfying all of the following properties: 
\begin{description}
\item [(i)] $|I_u| = \alpha r/40$.
\item [(ii)] $u \notin \bigcup_{i \in I_u} (V_i \cup V_{i+1})$.
\item [(iii)] At the beginning of Stage 2, $\deg^-_{D'}(u, V_i \setminus L_2) \geq \alpha \ell p/3$ and $\deg^+_{D'}(u, V_{i+1} \setminus L_2) \geq \alpha \ell p/3$ hold for every $i \in I_u$.
\item [(iv)] $(j - i) \mod r \geq 5$ and $(i - j) \mod r \geq 5$ for every $i \neq j \in I_u$.
\item [(v)] $u$ is not $i$-bad of type II for any $i \in I_u$. 
\end{description}
The existence of such a set follows from Lemma~\ref{lem::manyGoodPairs} and from Parts 1 and 4 of Proposition~\ref{prop::summary}. 

Note that, while Properties (i), (ii), (iv) and (v) hold throughout Stage 2, Property (iii) might be violated during the construction of $P_2$. Hence, we first prove that a.a.s., at the moment $|P_2 \setminus P_1| \geq \alpha \xi^3 n/10^5$ first occurs, $I_u$ satisfies a weaker version of this property.
\begin{claim} \label{cl::largeRemainingNeighbourhood}
With probability at least $1 - o(1/n)$, at the moment $|P_2 \setminus P_1| \geq \alpha \xi^3 n/10^5$ first occurs, $\deg^-_{D'}(u, V_i \setminus L_2) \geq \alpha \ell p/12$ and $\deg^+_{D'}(u, V_{i+1} \setminus L_2) \geq \alpha \ell p/12$ hold for every $i \in I_u$. 
\end{claim}

\begin{proof}
Fix an arbitrary $i \in I_u$ and let $N^-_{D'}(u, V_i \setminus L_2) \subseteq A_i \subseteq V_i \setminus L_2$ be an arbitrary set of size $2 \ell p$; such a set exists since $B \subseteq P_1$ and so $u \notin B$. It follows by Lemma~\ref{lem::threeStepsUpperBound} that whenever we perform 3 consecutive random forward steps $(x,y), (y,y')$ and $(y',z)$, where $z \in V_i$, we have $Pr(z \in A_i) \leq 3000 \xi^{-3} p$. Since we only consider the first $\alpha \xi^3 n/10^5$ random forward steps made in Stage 2, we consider at most $1 + \alpha \xi^3 n/(10^5 r)$ vertices of $V_i \cap (P_2 \setminus P_1)$. Let $Y_i \sim Bin(1 + \alpha \xi^3 n/(10^5 r), 3000 \xi^{-3} p)$, then $\mathbb{E}(Y_i) = (1 + \alpha \xi^3 n/(10^5 r)) \cdot 3000 \xi^{-3} p \leq \alpha \ell p/28$. We claim that, at the moment $|P_2 \setminus P_1| \geq \alpha \xi^3 n/10^5$ first occurs, $|A_i \cap (P_2 \setminus P_1)|$ is dominated by $Y_i$, that is, $Pr(|A_i \cap (P_2 \setminus P_1)| \geq K) \leq Pr(Y_i \geq K)$ holds for every $K$. Indeed, whenever we add to $P_2$ an arc $(u,v)$, where $u \in V_{i-1}$ and $v \in V_i$, we can imagine that a coin is tossed with the probability of \emph{success}, that is, the probability that $v \in A_i$, being at most $3000 \xi^{-3} p$. It thus follows by Theorem~\ref{th::Chernoff} (iv) that at the moment $|P_2 \setminus P_1| \geq \alpha \xi^3 n/10^5$ first occurs
\begin{eqnarray*} 
Pr(|N^-_{D'}(u, V_i) \cap (P_2 \setminus P_1)| \geq \alpha \ell p/4) &\leq& Pr(|A_i \cap (P_2 \setminus P_1)| \geq \alpha \ell p/4)\\ 
&\leq& Pr(Y_i \geq \alpha \ell p/4) \leq e^{- \alpha \ell p/4} \,.
\end{eqnarray*} 
An analogous argument shows that $Pr(|N^+_{D'}(u, V_{i+1}) \cap (P_2 \setminus P_1)| \geq \alpha \ell p/4) \leq e^{- \alpha \ell p/4}$ holds as well. A union bound over the $\alpha r/40$ elements of $I_u$ shows that the probability that there exists some $i \in I_u$ for which $|N^-_{D'}(u, V_i) \cap (P_2 \setminus P_1)| \geq \alpha \ell p/4$ or $|N^+_{D'}(u, V_{i+1}) \cap (P_2 \setminus P_1)| \geq \alpha \ell p/4$ is $o(1/n)$. It follows by Property (iii) that with probability at least $1 - o(1/n)$ we have $\deg^-_{D'}(u, V_i \setminus L_2) \geq \alpha \ell p/3 - \alpha \ell p/4 \geq \alpha \ell p/12$ and $\deg^+_{D'}(u, V_{i+1} \setminus L_2) \geq \alpha \ell p/12$ for every $i \in I_u$.
\end{proof}  

Consider the path $P_2$ at the moment $|P_2 \setminus P_1| \geq \alpha \xi^3 n/10^5$ first occurs and assume that $\deg^-_{D'}(u, V_i \setminus L_2) \geq \alpha \ell p/12$ and $\deg^+_{D'}(u, V_{i+1} \setminus L_2) \geq \alpha \ell p/12$ hold at this point for every $i \in I_u$. At this point, let $X_u = |\{(x,y) \in E(P_2) : \exists i \in I_u \textrm{ such that } x \in N^-_{D'}(u, V_i) \textrm{ and } y \in N^+_{D'}(u, V_{i+1})\}|$, that is, $X_u$ is a random variable which counts some of the arcs of $P_2$ which can absorb $u$. In order to complete the proof of Lemma~\ref{goodArcs}, it suffices to prove that $Pr(X_u < 10^{-10} \alpha^4 \xi^3 p^2 n) = o(1/n)$. Let $i \in I_u$ be an arbitrary index. Let $(x_1,x_2), (x_2,x_3), (x_3,x_4)$ and $(x_4,x_5)$ be four consecutive arcs of $P_2 \setminus P_1$, where $x_m \in V_{i+m-4}$ for every $1 \leq m \leq 5$. Assume that $|P_2 \setminus P_1| \leq \alpha \xi^3 n/10^5$ was still true immediately after the random forward step $(x_4,x_5)$ was made. It follows by the description of Algorithm 3, by~\eqref{eq::smallL1} and by Parts 2 and 3 of Proposition~\ref{prop::summary} that $|V_i \cap L_2| \leq 2 \ell/3$ holds at this point for every $1 \leq i \leq r$. It follows by Properties (ii) and (v) above and by Claim~\ref{cl::largeRemainingNeighbourhood} that there exist sets $A_u^- \subseteq N^-_{D'}(u, V_i \setminus L_2)$ of size $|A_u^-| \geq (1 - \varepsilon') |N^-_{D'}(u, V_i \setminus L_2)| \geq \alpha \ell p/13$ and $A_u^+ \subseteq N^+_{D'}(u, V_{i+1} \setminus L_2)$ of size $|A_u^+| \geq (1 - \varepsilon') |N^+_{D'}(u, V_{i+1} \setminus L_2)| \geq \alpha \ell p/13$ for which all the conditions of Lemma~\ref{lem::fourSteps} are satisfied (with $Z_1 = A_u^-$, $Z_2 = A_u^+$, $X = L_2$, $q_1 = \lambda \ell p$ and $q_2 = 2 \ell p$). Therefore 
\begin{eqnarray} \label{eq::oneArc}
Pr(x_4 \in N^-_{D'}(u, V_i) \textrm{ and } x_5 \in N^+_{D'}(u, V_{i+1})) &\geq& Pr(x_4 \in A_u^- \textrm{ and } x_5 \in A_u^+) \nonumber\\
&\geq& \frac{|A_u^-| |A_u^+|}{2 \ell^2} \geq \frac{(\alpha \ell p/13)^2}{2 \ell^2} \geq \alpha^2 p^2/400 \,.  
\end{eqnarray} 

Let $Y_u^j : j \in I_u$ be independent random variables, where $Y_u^j \sim Bin \left(\frac{\alpha \xi^3 n}{2 \cdot 10^5 r}, \frac{\alpha^2 p^2}{400} \right)$ for every $j \in I_u$. Let $Y_u = \sum_{j \in I_u} Y_u^j$, then $Y_u \sim Bin \left(\frac{\alpha^2 \xi^3 n}{8 \cdot 10^6}, \frac{\alpha^2 p^2}{400} \right)$. We claim that $X_u$ dominates $Y_u$, that is, that $Pr(X_u < K) \leq Pr(Y_u < K)$ for every $K$. Indeed, note that the inequality~\eqref{eq::oneArc} holds regardless of the choice of $x_1$ (as long as it is a nice vertex, it is not bad of type I with respect to $V_{i-2}, V_{i-1}$ and $i \in I_u$). Therefore, whenever we add to $P_2$ an arc $(x,y)$, where $x \in V_j$ and $y \in V_{j+1}$ for some $j \in I_u$, we can imagine that a coin is tossed with the probability of \emph{success}, that is, the probability that $x \in N^-_{D'}(u, V_j)$ and $y \in N^+_{D'}(u, V_{j+1})$, being at least $\alpha^2 p^2/400$. Moreover, for every $j \in I_u$, we consider all arcs $(x,y) \in E_{D'}(V_j, V_{j+1})$, added to $P_2$ during Stage 2 until $|P_2 \setminus P_1| \geq \alpha \xi^3 n/10^5$ first occurred. Hence we consider at least $\frac{\alpha \xi^3 n/10^5}{r} - 1 \geq \frac{\alpha \xi^3 n}{2 \cdot 10^5 r}$ arcs $(x,y)$ such that $x \in V_j$ and $y \in V_{j+1}$, that is, there are at least $\frac{\alpha \xi^3 n}{2 \cdot 10^5 r} \cdot |I_u| = \frac{\alpha^2 \xi^3 n}{8 \cdot 10^6}$ trials. Hence, it follows by Theorem~\ref{th::Chernoff} (i) that
\begin{eqnarray*}
Pr \left(X_u < 10^{-10} \alpha^4 \xi^3 p^2 n \right)
&\leq& Pr \left(Y_u < 10^{-10} \alpha^4 \xi^3 p^2 n \right)
\leq Pr(Y_u \leq \mathbb{E}(Y_u)/2) \\
&\leq& \exp \left\{- \frac{1}{8} \cdot \frac{\alpha^4 \xi^3 p^2 n}{32 \cdot 10^8} \right\}
= o(1/n) \,.
\end{eqnarray*}

Since $u \in V \setminus P_2$ was arbitrary, it follows by a union bound argument that a.a.s. Claim~\ref{cl::largeRemainingNeighbourhood}, and thus also $X_u < 10^{-10} \alpha^4 \xi^3 p^2 n$, hold for every $u \in V \setminus P_2$.    
\end{proof}

\subsection{Stages 3 and 4: Closing the path into a cycle and absorbing all remaining vertices} \label{subsec::stage34}

This subsection consists of two parts, namely Stage 3 and Stage 4. In Stage 3 we will close the path $P_2$ which was built in Stage 2 into a cycle. In Stage 4 we will use Lemma~\ref{goodArcs} and Lemma~\ref{lem::matchingStage4} to absorb all of the remaining vertices into $C$ thus creating a Hamilton cycle. 

\textbf{Stage 3:} In this stage we close $P_2$ into a directed cycle $C$, by adding a few more arcs. Throughout this stage, we denote the current path by $P_3$ and let $L_3 = T_1 \cup P_3$. Initially $P_3 = P_2$. 

\begin{algorithm}
\caption{Close $P_2$ into a cycle $C$}
\begin{algorithmic} 
\STATE $x \leftarrow$ last vertex added to $P_2$ in Stage 2
\WHILE{$i(x) \neq r-3$}
\STATE $x \leftarrow$ standard forward step from $x$ with respect to $L_3 \cup A_0$
\ENDWHILE
\STATE make closing step from $x$ to $v_0$ with respect to $L_3$. 
\end{algorithmic}
\end{algorithm}

\bigskip

It is evident that, if it works, this algorithm returns a cycle $C$ of $D'$. It thus remains to prove the correctness of the algorithm. The existence of the required standard forward steps follows from Lemma~\ref{lem::stepsExist} since the last vertex added to $P_2$ in Stage 2 was nice, $T_1 \subseteq L_3$ and $|V_i \setminus L_3| \geq |V_i \setminus L_2| - 2 > 2 \varepsilon' \ell$ holds for every $1 \leq i \leq r$ throughout this stage. Since, moreover, $A_0 \subseteq N_{D'}^-(v_0, V_r \setminus L_3)$ is of size $\lambda \ell p$, the conditions of Lemma~\ref{lem::closingStep} are met as well; this proves the existence of the required closing step.

\textbf{Stage 4:} 
In this final stage, we extend the directed cycle $C$, built in Stage 3, to a Hamilton cycle of $D'$ by absorbing all remaining vertices. Let $t = |V \setminus V(C)|$ and let $H$ denote 
the bipartite graph with bipartition $V(H) = (V \setminus V(C)) \cup E(C)$ in which a vertex $u \in V \setminus V(C)$ is connected by an edge of $H$ to an arc $(x,y) \in E(C)$ if and only 
if $(x,u) \in E(D')$ and $(u,y) \in E(D')$.

\begin{algorithm}[H]
\caption{Extend $C$ to a Hamilton cycle}
\begin{algorithmic} 
\STATE ${\mathcal M} \leftarrow \{\{u_j, (x_j, y_j)\} : 1 \leq j \leq t\}$ a matching of $H$
\FOR{$j=1$ \TO $t$}
\STATE $C \leftarrow (C \setminus \{(x_j, y_j)\}) \cup \{(x_j, u_j), (u_j, y_j)\}$ 
\ENDFOR
\end{algorithmic}
\end{algorithm}

\bigskip

It is evident that, if it works, this algorithm returns a Hamilton cycle of $D'$. Thus, in order to complete the proof of Theorem~\ref{directedHam} it remains to prove that $H$ admits a matching of size $t$. It follows by the description of Stage 2 that $\left||V_i \cap (P_2 \setminus P_1)| - |V_j \cap (P_2 \setminus P_1)|\right| \leq 1$ holds for every $1 \leq i,j \leq r$. 
Hence, at the end of Stage 2 we have $3 \varepsilon' \ell \leq |V_s \setminus L_2| \leq 1 + 3 \varepsilon' \ell + 
\max \{|V_i \cap P_1| : 1 \leq i \leq r\} \leq 4 \varepsilon' \ell$ for every $1 \leq s \leq r$, where the last inequality holds by Part 3 of Proposition~\ref{prop::summary} and since $\rho \ll \varepsilon'$. Therefore 
\begin{eqnarray} \label{eq::smallMatching}
t &\leq& |V \setminus \bigcup_{i=1}^r V_i| + \sum_{i=1}^r (|V_i \setminus L_2| + |V_i \cap T_1| + |V_i \cap A_0|) \leq \varepsilon n + 2 \sqrt{\varepsilon}k \cdot \ell + r (4 \varepsilon' \ell + 2 \rho \ell + \lambda \ell p) \nonumber \\ 
&\leq& 10 \sqrt{\varepsilon'} n \leq 10^{-11} \alpha^4 \xi^3 n \,,
\end{eqnarray}
where the last inequality holds since $\varepsilon' \ll \alpha^8 \xi^6$.

Since every non-empty subset of $E(C)$ spans a digraph with maximum out-degree 1 and maximum in-degree 1, it follows by~\eqref{eq::smallMatching} and by Lemma~\ref{lem::edgeDistribution} (with $\alpha = t/n$ and $\beta = 10^{-10} \alpha^4 \xi^3$) that a.a.s. $e_H(X,Y) < 10^{-10} \alpha^4 \xi^3 p^2 n |X|$ for every $X \subseteq V \setminus V(C)$ and $Y \subseteq E(C)$ such that $|X| = |Y|$. Moreover, it follows by Lemma~\ref{goodArcs} that a.a.s. $\deg_H(u) \geq 10^{-10} \alpha^4 \xi^3 p^2 n$ holds for every $u \in V \setminus V(C)$. We conclude that a.a.s. $H$ satisfies the conditions of Lemma~\ref{lem::matchingStage4} (with $A = V \setminus V(C)$, $B = E(C)$ and $\delta = 10^{-10} \alpha^4 \xi^3 p^2 n$) and thus a.a.s. there exists a matching of $H$ which saturates $V \setminus V(C)$.

\section{Concluding remarks and open problems}
\label{sec::openprob}

We have proved that a.a.s. $(1/2 - \alpha) n p \leq r_{\ell}({\mathcal D}(n,p), {\mathcal H}) \leq (1/2 + \alpha) n p$, where $\alpha > 0$ is an arbitrarily small constant, provided that $p \gg \log n/ \sqrt{n}$. For \emph{undirected} random graphs it was proved in~\cite{LS} that $r_{\ell}(G(n,p), {\mathcal H}) = (1/2 + o(1)) n p$ holds a.a.s. for every $p \gg \log n/n$. This is essentially tight since for $p < \log n/n$ a.a.s. $G(n,p)$ contains no Hamilton cycle. Since it is also known (see~\cite{McDiarmid} and~\cite{Frieze}) that for $p = \Omega(\log n/n)$ directed random graphs are a.a.s. Hamiltonian, it is natural to ask the following question.

\begin{question} 
\label{sparseRandomDigraph} 
Is it true that for $p \gg \log n/n$, a.a.s. every subdigraph of ${\mathcal D}(n,p)$ with minimum out-degree and in-degree at least $(1/2 + o(1)) n p$ contains a directed Hamilton cycle?
\end{question}

Recall that our proof of the upper bound in Theorem~\ref{directedHam} does hold for every $p \gg \log n/n$. On the other hand, since our proof method for the lower bound relies on the existence of linearly many pairwise arc disjoint triangles in ${\mathcal D}(n,p)$, each sharing one arc with a given cycle (recall Stage 4), it cannot be used when $p = o(n^{- 1/2})$, and hence some new ideas are required.

\vspace{0.25cm}
\noindent
{\bf Acknowledgment.} We would like to thank Asaf Ferber for many stimulating discussions.

\end{document}